\documentclass[11pt]{amsart}
\usepackage{amscd}
\usepackage{amsfonts}
\usepackage{amssymb,latexsym}
\usepackage{enumerate}
\usepackage{amsmath}
\usepackage{mathrsfs}
\usepackage{amsthm}
\usepackage{hyperref}
\usepackage[square, comma, sort&compress, numbers]{natbib}
\usepackage{bm}
\usepackage{esint}
\usepackage{fancyhdr}
\usepackage{lastpage}
\usepackage{layout}
\usepackage{ulem}
\usepackage{extarrows}
\newcommand{\ol}{\overline}

\newcommand{\pa}{\partial}
\newcommand{\vp}{\varphi}
\DeclareMathOperator\tr{tr}
\DeclareMathOperator\id{id}
\DeclareMathOperator\rank{rank}
\DeclareMathOperator\Ric{Ric}

\DeclareMathOperator\dvol{dvol}
\DeclareMathOperator\End{End}
\DeclareMathOperator\loc{loc}

\DeclareMathOperator\Aut{Aut}
\begin{document}
\newcounter{remark}
\newcounter{theor}
\setcounter{remark}{0}
\setcounter{theor}{1}
\newtheorem{claim}{Claim}[section]
\newtheorem{theorem}{Theorem}[section]
\newtheorem{lemma}{Lemma}[section]
\newtheorem{corollary}{Corollary}[section]
\newtheorem{corollarys}{Corollary}
\newtheorem{proposition}{Proposition}[section]
\newtheorem{question}{question}
\newtheorem{defn}{Definition}[section]
\newtheorem{examp}{Example}[section]
\newtheorem{assumption}{Assumption}[section]
\newtheorem{rem}{Remark}[section]
\newtheorem*{theorem1}{Theorem}
\numberwithin{equation}{section}
\title{Poisson metrics and Higgs bundles over noncompact K\"{a}hler manifolds}
\author{Di Wu}
\address{Di Wu, School of Mathematical Sciences, University of Science and Technology of China, Hefei 230026, People's Republic of China}
\email{wudi123@mail.ustc.edu.cn}
\author{Xi Zhang}
\address{Xi Zhang, School of Mathematical Sciences, University of Science and Technology of China, Hefei 230026, People's Republic of China}
\email{mathzx@ustc.edu.cn}
\thanks{The research was partially supported by the project "Analysis and Geometry on Bundle" of Ministry of Science and Technology of the People's
Republic of China, No. SQ2020YFA070080. The authors are partially supported by NSF in China No. 11625106, 11801535 and 11721101.}
\maketitle
\begin{abstract}
In this paper, we study the existence of Poisson metrics on flat vector bundles over noncompact Riemannian manifolds and discuss related consequence, specially on the applications in Higgs bundles,
towards generalizing Corlette-Donaldson-Hitchin-Simpson's nonabelian Hodge correspondence to noncompact K\"{a}hler manifolds setting.
\end{abstract}
\section{Introduction}
\subsection{Background}
Let $(E,D)$ be a vector bundle over a Riemannian manifold $(M,g)$, we say $(E,D)$ is simple if it admits no proper $D$-invariant sub-bundle
and $(E,D)$ is semi-simple\footnote{In literature, one sometimes says that complete reducible or reductive.} if it is a direct sum of $(E_i,D_i)$ and
each of them is simple. The connection $D$ is called irreducible if it admits no nontrivial covariantly constant section of $\End(E)$. Given any metric $H$ on $E$, it splits the connection as
\begin{equation}\begin{split}\label{decomposition}
D=D_H+\psi_H,
\end{split}\end{equation}
where $D_H$ is a connection preserving $H$ and $\psi_H\in\Omega^1(\End(E))$ is self-adjoint with respect to $H$. We say $H$ is harmonic if it satisfies
\begin{equation}\begin{split}\label{harmonicequation}
D_H^\ast\psi_H=0,
\end{split}\end{equation}
where $D_H^\ast$ is the adjoint of $D_H$. The notion of harmonic map originated from Eells-Sampson \cite{ES1964}
and was generalized to the notion of harmonic metric by Corlette \cite{Co1988}. The motivation for this name may be explained as follows.
By the Riemann-Hilbert correspondence, suppose $\rho_D:\pi_1(M)\rightarrow GL(r)$ bs the representation corresponding to a flat bundle $(E,D)$,
then any metric $H$ on $E$ induces a $\rho_D$-equivariant map $f_H$ from $\tilde{M}$ to $GL(r)/U(r)$, where $\tilde{M}$ is the universal of $M$.
Then we can see that $H$ being harmonic is equivalent to $f_H$ being a harmonic map or $df_H$ being a harmonic $1$-form with values in $f_H^{-1}T(GL(r)/U(r))$.
\par The study of harmonic metrics is a fundamental problem which remains an active branch of current research in differential geometry,
it builds a bridge between geometry analysis and many other areas including group representations, nonabelian Hodge theory.
Let $(M,g)$ be a compact Riemannian manifold, the main existence result is the following
\begin{theorem}[Donaldson \cite{Do1987},Corlette \cite{Co1988}]
A flat bundle admits a harmonic metric if and only if it arises from a semi-simple representation of the fundamental group $\pi_1(M)$.
\end{theorem}
If $(M,g)$ happens to be a compact K\"{a}hler manifold, the significance of this theorem lies in the construction of Higgs bundle. In fact, the pair
$(D_H^{0,1},\psi_H^{1,0})$ will determine a poly-stable Higgs structure, provided $H$ being harmonic. For this reason, sometimes we then
say that $D$ satisfies the Hitchin's self-duality equation \cite{Hi1987} and we call $D$ a Hitchin connection.
Combing this with the Donaldson-Uhlenbeck-Yau theorem(\cite{Do1985,UY1986}) for Higgs bundles \cite{Hi1987,Si1988}, one obtains
Corlette-Donaldson-Hitchin-Simpson's nonabelian Hodge correspondence:
\begin{theorem}[\cite{Co1988,Do1987,Hi1987,Si1988}]
On a compact K\"{a}hler manifold, there is an one to one correspondence between the category of seme-simple flat bundles and the category of poly-stable Higgs bundles with vanishing Chern numbers.
\end{theorem}
\par This identification can be traced back to Narasimhan-Seshadri \cite{NS1965} and has yielded serval fascinating topological results, specially on the fundamental groups of compact K\"{a}hler manifolds, see \cite{Co1993,Si1992} for more details. In \cite{JZ1996,JZ1997}, Jost-Zuo studied the representations of fundamental groups for quasi-projective varieties by using the theory of harmonic metrics on quasi-projective varieties.
Besides that, there are also some partial existence results for harmonic metrics on noncompact manifolds. Simpson \cite{Si1990}
obtained the existence of harmonic metrics on punctured Riemannian surface in his study of nonabelian Hodge theory over noncompact curves.
From the view point of equivariant harmonic maps, Li \cite{Li1996} considered the case that $M=\ol{M}-D$ equipped with a Poincar\'{e}-like
metric $g$, where $\ol{M}$ is a compact complex manifold and $D$ be a divisor with normal crossings in $\ol{M}$.
Recently, Mochizuki makes a significant process towards the generations of the Kobayashi-Hitchin correspondence and the nonabelian Hodge correspondence, the existence of pluriharmonic metrics for good filtered flat bundles over smooth projective varieties play an important role,
see \cite{Mo2006,Mo2007,Mo2009,Mo2011}. There are many interesting and important literature concerning the existence of harmonic metrics
and related topics, see \cite{BGM2020,BK2019,Co1992,DMG2018,JY1991,La1991,PZZ2019,Zh2021} and the references therein.
\par Flat bundles occur in great abundance since it corresponds local systems and representations of fundamental groups.
It was observed by Loftin \cite{Lo2009} that for an affine manifold $M$, any flat bundle $(E,D)$ over $M$ corresponds a holomorphic
bundle $\tilde{E}$ over the tangent bundle $TM$. And then the classical Donaldson-Uhlenbeck-Yau theorem can be generalized to the affine case,
see \cite{BLK2013,Lo2009,SZZ2019}. Recently, Collins-Jacob-Yau \cite{CJY2019} found that the dimension reduction of the
Hermitian-Yang-Mills equation on $\tilde{E}$ is given by the Poisson metric equation on $(E,D)$. Here we say $H$ is a Poisson metric if
\begin{equation}\begin{split}\label{poissonequation}
D_H^\ast\psi_H=c_H\id_E,
\end{split}\end{equation}
where $c_H$ is a function on $M$. Motivated by this and the understanding of stable vector bundles on $K3$ surface in the large complex structure limit, Collins-Jacob-Yau \cite{CJY2019} focused on the research of Poisson metrics on flat vector bundles. Under some suitable assumptions, they related the existence of Poisson metrics on flat bundles with parabolic structures over punctured Riemann surfaces to certain notion of slope stability.
\subsection{Main results}
In the first part of this paper, we shall attack the existence of Poisson metrics on flat bundles over noncompact Riemannian manifolds of
arbitrary dimension. Below we state our first theorem, which can be interpreted as a noncompact version of the result of Donaldson and Corlette.
\begin{theorem}\label{thm1}
Let $(M,g)$ be a Riemannian manifold(not necessarily compact) satisfying the Assumptions \ref{assump1} and either Assumption
\ref{assump2} or complete with finite volume, $E$ be a vector bundle over $M$ equipped with a connection $D$ and a metric $K$ such that $|D_K^\ast\psi_K|\leq C\phi$ for a constant $C$. Then it admits a unique Poisson metric $H$ such that $\det h=1$, $|h|\in L^\infty$ and $|Dh|\in L^2$ for $h=K^{-1}H$, provided $D$ being flat
and $K$-analytically stable. Conversely, if $D$ is irreducible, it admits such a Poisson metric only if $D$ is $K$-analytically stable.
\end{theorem}
For the existence statement, we actually only need the Assumption \ref{assump1} and the condition $|(D_K^\ast\psi_K)^\perp|\leq C\phi$, where
$(D_K^\ast\psi_K)^\perp$ is the trace-free part of $D_K^\ast\psi_K$, see Proposition \ref{poissonexistence} for details. The Assumption \ref{assump1} is introduced by Mochizuki \cite{Mo2020} for K\"{a}hler manifolds and the volume can be infinite. The notion of analytically stability in Theorem \ref{thm1} will be introduced in Section 2.2. If $M$ is compact, we know that $D$ being analytically stable(poly-stable) if and only if it is simple(semi-simple), while the correct substitution should be replaced by analytically stability in noncompact setting. Moreover, if $(M,g)$ be a K\"{a}hler manifold and $D$ being flat, the definition coincides with that in Simpson's paper \cite{Si1990}, see Remark \ref{kahlerstability}.
\par The second statement in Theorem \ref{thm1} indicates that $D$ is analytically stable, not only with respect to the Poisson metric,
but also the background metric $K$. It may be mentioned that it is facilitated by the feature of the stability of vector bundles, while the
issue is more complicated in the setting of holomorphic bundles. We will use Mochizuki's exhaustion method in \cite{Mo2020}, where he proved
the existence of exhaustion function $f$. Denote by $M_s$ the set where $f\leq s$ and we firstly consider the parabolic heat flow
$(\ref{flow1})$ on $M_s$, which can be regard as the vector bundle analogy of Eells-Sampson's flow for harmonic maps or Donaldson's heat flow \cite{Do1985,Do1992}. For arbitrary $D$, we show the existence and uniqueness of the long-time solution to $(\ref{flow1})$, also the Dirichlet boundary problems for harmonic metric and Poisson metric are uniquely solvable, see Proposition \ref{Dirichlet1} and Proposition \ref{Dirichlet2}.
Next we observe in Proposition \ref{keyproposition} that the crucial identity $(\ref{outlinekeyidentity})$, based on which serval core a prior estimates are established. Once the zeroth estimate of a sequence of the solutions to the Dirichlet boundary problem is obtained,
the evolved metrics are completely controlled and will converge to a Poisson metric. Finally, we prove the uniqueness and the second statement by comparing the analytically stabilities with respect to different reference fiber metrics.
\par We mention that Corlette \cite{Co1988} studied the following heat flow
\begin{equation}\begin{split}
\frac{\pa D_t}{\pa t}=-D_tD_{t,K}^\ast \psi_{D_t,K},
\end{split}\end{equation}
on flat complex vector bundles $(E,D_0)$ over compact Riemannian manifolds, where $D_t$ is a time-dependent family of connections
in the orbit of $D_0$. Corlette proved the convergence of $D_t$ at $t=\infty$ with the property $D_{\infty,K}\psi_{D_\infty,K}=0$ and
if $(E,D_0)$ is simple, $D_\infty$ must lie in the complex gauge orbit of $D_0$. Corlette's proof relies on Uhlenbeck's deep compactness
theorem \cite{Uh1982} and cannot be generalized to noncompact setting directly. Even in compact case, here the discussion is different
from that in \cite{Co1988} and we bypass Uhlenbeck's compactness theorem.
\par As a corollary of $(\ref{DHstarpsiHt})$ and Theorem \ref{thm1}, we have
\begin{corollary}\label{cor1}
Under the same assumptions in Theorem \ref{thm1} and $\tr D_K^\ast\psi_K=0$, there exist a unique harmonic metric $H$
such that $\det h=1$, $|h|\in L^\infty$ and $|Dh|\in L^2$ for $h=K^{-1}H$, provided $D$ being flat and $K$-analytically stable.
Conversely, $D$ is $K$-analytically stable if it is irreducible and admits such a harmonic metric.
\end{corollary}
\par By $(\ref{DHKstarpsiHK})$, we know that
\begin{equation}\begin{split}
\tr D^\ast_H\psi_H-\tr D^\ast_K\psi_K
&=-\frac{1}{2}\tr D_K^\ast(h^{-1}\delta_Kh)
=\frac{1}{2}\Delta\log\det h,
\end{split}\end{equation}
where $h=K^{-1}H$ and $\delta_K=D_K-\psi_K$. So by solving a scalar Poisson equation(Theorem 4.3 in \cite{WZ2011}) and employing Theorem \ref{thm1},
it is easy to conclude
\begin{corollary}\label{cor2}
Let $(M,g)$ be a Riemannian manifold satisfying the Assumption \ref{assump1} and Assumption \ref{assump2'} with $\phi=1$, $E$ be a vector bundle
over $M$ equipped with a connection $D$ and a metric $K$ such that $|D_K^\ast\psi_K|\in L^\infty$, $|\psi_K|\in L^2$. Then it admits a unique harmonic metric $H$ such
that $\det h=1$, $|h|\in L^\infty$ and $|Dh|\in L^2$ for $h=K^{-1}H$, provided $D$ being flat and $K$-analytically stable. Conversely,
$D$ is $K$-analytically stable if it is irreducible and admits such a harmonic metric.
\end{corollary}
There is a remark on the situation that there is no stability(simple) assumption. For an arbitrary connection on the vector bundle over a compact Riemannian manifold, there is no obstruction to the existence of $L^\infty$-approximate solutions of $(\ref{harmonicequation})$, see Proposition \ref{approximateharmonic}.
\par In the second part of this paper, we shall discuss related applications and consequences after Theorem \ref{thm1}. Recall a Higgs bundle
$(E,\ol\pa_E,\theta)$ over a complex manifold $X$ consists of a holomorphic vector bundle $(E,\ol\pa_E)$ over $X$ together with holomorphic section
$\theta\in\Omega^{1,0}(\End(E))$ such that $\theta\wedge\theta=0$. For a Hermitian metric $K$, we consider the
Hitchin-Simpson connection
\begin{equation}\begin{split}
D_{\ol\pa_E,\theta,K}=\pa_{\theta,K}+\ol\pa_{E,\theta}, \pa_{\theta,K}=\pa_K+\theta^{\ast K}, \ol\pa_{E,\theta}=\ol\pa_E+\theta,
\end{split}\end{equation}
where $D_{\ol\pa_E,K}=\pa_K+\ol\pa_E$ is the Chern connection. If the curvature $F_{\ol\pa_E,\theta,K}=D_{\ol\pa_E,\theta,K}^2$ vanishes, we will call
$(\ol\pa_E,\theta,K)$ a Higgs flat structure on the underlying bundle $E$. In particular, if the Chern curvature $F_{\ol\pa_E,K}=D_{\ol\pa_E,K}^2$
vanishes(that is, $\theta=0$), we then say $(\ol\pa_E,K)$ a flat holomorphic structure.
\par As an application of Theorem \ref{thm1}, we have
\begin{theorem}\label{thm2}
Let $(X,\omega)$ be a K\"{a}hler manifold satisfying the Assumption \ref{assump1} and $(E,D)$ be a $K$-analytically stable flat bundle over $X$
with $|(D_K^\ast\psi_K)^\perp|\leq C\phi$ for a constant $C$, then
\begin{enumerate}
\item If $(X,\omega)$ is a complex curve, there is a Higgs flat structure on $E$.
\item If $(X,\omega)$ is complete with bounded Ricci curvature from below and $|\psi_K|\in L^2$, there is a Higgs flat
structure on $E$.
\item If $(X,\omega)$ is complete with nonnegative Ricci curvature and $|\psi_K|\in L^2$, there is a flat holomorphic structure on $E$.
\end{enumerate}
\end{theorem}
Especially, the third statement of the conclusions in Theorem \ref{thm2} indicates there is a metric compatible flat connection on $E$.
Finding metric compatible flat connections may be of interest in its own right. By Corlette's work in \cite{Co1988},
it is easy to see the same conclusion also holds for a flat bundle over a compact Riemannian manifold with nonnegative Ricci curvature and
so it follows that the Euler class $e(E)$ of $E$ is zero.
\par In order to prove Theorem \ref{thm2}, we need to construct Higgs structures by employing Theorem \ref{thm1}, see Proposition
\ref{Higgsstructure1} and Proposition
\ref{Higgsstructure2}, which is a key step to establish the Corlette-Donaldson-Hitchin-Simpson's correspondence. If $X$ is compact and $H$ being harmonic,
this is equivalent to determine whether $H$ being pluriharmonic, the issue does not appear and it is essentially due to Siu-Sampson \cite{Sa1986,Si1980}
 and Corlette \cite{Co1988}, see also \cite{Si1992}. If $X$ is noncompact, see \cite{JZ1997} for case that the complement of a divisor in a compact
 K\"{a}hler manifold.
\par Secondly, we will focus on generalizing the Corlette-Donaldson-Hitchin-Simpson's correspondence to noncompact K\"{a}hler manifolds setting. Following Simpson \cite{Si1988}, if $(\ol\pa_E,\theta)$ is a Higgs structure on $E$, the $K$-analytically degree is defined by
\begin{equation}\begin{split}
\deg_\omega(E,K)=\int_X\sqrt{-1}\tr\Lambda_\omega F_{\ol\pa_E,\theta,K}\dvol_\omega.
\end{split}\end{equation}
If $\mathcal{V}\subseteq\mathcal{O}_X(E)$ is a sub-sheaf, then it can be seen as a sub-bundle $V$ outside a singular
set $\Sigma_{\mathcal{V}}$ of codimension at least two. Let $\pi_K$ denotes the projection onto $V$ and $K$ restricts to a metric
on $V$ so that we can define $\deg_\omega(\mathcal{V},K)$ by integrating outside $\Sigma_\mathcal{V}$. We have the Chern-Weil formula written in terms of $\Lambda_\omega F_{\ol\pa_E,\theta,K}$,
\begin{equation}\begin{split}\label{ChernHiggs}
\deg_\omega(\mathcal{V},K)=\int_{X\setminus\Sigma_\mathcal{V}}\left(\sqrt{-1}\tr(\pi_K\circ\Lambda_\omega F_{\ol\pa_E,\theta,K})-|\ol\pa_{E,\theta}\pi_{K}|_K^2\right)\dvol_\omega.
\end{split}\end{equation}
$(\ol\pa_E,\theta)$ is called $K$-analytically stable if for any proper sub-Higgs sheaf $\mathcal{V}\subset\mathcal{O}_X(E)$, it holds
\begin{equation}\begin{split}
\frac{\deg_\omega(\mathcal{V},K)}{\rank(\mathcal{V})}<\frac{\deg_\omega(E,K)}{\rank(E)}.
\end{split}\end{equation}
\par From now on, we fix a Hermitian metric $K$ on the vector bundle $E$,  write
\begin{itemize}
\item $\mathcal{M}_{Flat,K}$: the set of isomorphic classes of $K$-analytically stable and irreducible
flat structures such that $|(D_K^\ast\psi_K)^\perp|\in L^1$.
\item $\mathcal{M}_{Higgs,K}$: the set of isomorphic classes of $K$-analytically stable and irreducible
Higgs structures with $c_1(E,K)=0$ and $|\Lambda_\omega F_{\ol\pa_E,\theta,K}|\in L^1$.
\item $\mathcal{M}_{Flat,K,\phi}$: the set of isomorphic classes of $K$-analytically stable and irreducible
flat structures such that $|(D_K^\ast\psi_K)^\perp|\leq C\phi$ for a constant $C$.
\item $\mathcal{M}_{Higgs,K,\phi}$: the set of isomorphic classes of $K$-analytically stable and irreducible
Higgs structures with $c_1(E,K)=0$ and $|\Lambda_\omega F_{\ol\pa_E,\theta,K}|\leq C\phi$ for a constant $C$.
\end{itemize}
\par In the above, isomorphic means that they are in the the same $\mathcal{G}$-orbit, where $\mathcal{G}$ denotes the set of $L^\infty$-sections of
$\Aut(E)$. We say a Higgs structure $(\ol\pa_E,\theta)$ is irreducible if $D_{\ol\pa_E,\theta,H}$ is irreducible, for any Hermitian-Einstein metric $H$ with $\det h=1$ and $|h|\in L^\infty$ for $h=K^{-1}H$.
\par Our third theorem can be stated as the follows.
\begin{theorem}\label{thm3}
Assume $(X,\omega)$ be a noncompact complex curve satisfying the Assumption \ref{assump1} and either Assumption
\ref{assump2} or complete with finite volume, then we have
\begin{enumerate}
\item There is map from $\mathcal{M}_{Flat,K,\phi}$ to $\mathcal{M}_{Higgs,K}$.
\item There is map from $\mathcal{M}_{Higgs,K,\phi}$ to $\mathcal{M}_{Flat,K}$.
\end{enumerate}
\end{theorem}
\par Notice that the complex plane $\mathbb{C}$ and compact Riemann surfaces punctured finite points satisfy the assumptions in Theorem \ref{thm3}. Moreover, the objects are analytically stable with respect to the same metric $K$. For higher dimensional case, under some extra technical assumptions(see \cite{Si1988,Mo2020}), the direction that from $K$-analytically stable and irreducible Higgs structures to $K$-analytically
stable and irreducible flat structures can be finished by the results of Simpson \cite{Si1988}, Mochizuki \cite{Mo2020} and the discussion as
that in $(2)$ of Theorem \ref{thm3}. Conversely, start with a $K$-analytically stable and irreducible flat structure,
using Proposition \ref{Higgsstructure2}(or Proposition \ref{Higgsstructure3}) instead, we can obtain an irreducible Higgs structure. In the process, the analytic estimates in Theorem \ref{thm3} remain valid, but currently it is not clear that when the resulting Higgs structure is stable with respect to $K$. We will further discuss this problem in the future.
\par Theorem \ref{thm3} is established based on the existence of Poisson metrics(Theorem \ref{thm1}) and Hermitian-Einstein metrics(\cite{Mo2020,Si1988}), while analytic estimates and the comparison of different analytically stabilities play the significant roles. If $X$ is compact, Theorem \ref{thm3} and the corresponding result for high dimensional case reduce to the classical nonabelian Hodge correspondence on compact K\"{a}hler manifolds. This correspondence has been generalised to parabolic and noncompact setting, specially on the case of a projective variety equipped with normal crossing divisors, in a series of works including those of Simpson \cite{Si1990} and Biquard-Boalch \cite{BB2004} for the tame and wild case in complex dimension one. And the higher dimensional case can be found in \cite{Bi1997,Mo2006,Mo2009,Mo2011}.
\subsection{Organization}
The rest of this paper is organized as follows. In Section 2, we discuss the assumptions on the base manifolds, introduce the notation of stability of vector bundles, and present some useful calculations. In Section 3, we solve the Dirichlet boundary problems for harmonic metric equation and Poisson metric equation, then give the proof of Theorem \ref{thm1}. In Section 4, we firstly investigate the uniqueness of Poisson metrics and then prove Theorem \ref{thm2}, Theorem \ref{thm3}. As a byproduct, also a vanishing theorem of Kamber-Tondeur classes is obtained.
\section{Preliminaries}
\subsection{Assumptions on the base spaces}
Assume $(M,g)$ be a Riemannian manifold, we list the following assumptions that needed in this paper.
\begin{assumption}\label{assump1}
There is a function $\phi:M\rightarrow\mathbb{R}_{\geq0}$ with $\phi\in L^1$, such that if $f$ is a nonnegative bounded function on $M$ satisfying
$\Delta f\geq-B\phi$ for a positive constant $B$ in distribution sense, we have
\begin{equation}\begin{split}\label{sec211}
\sup\limits_{M}f\leq C(B)(1+\int_Mf\phi\dvol_g).
\end{split}\end{equation}
Furthermore, if $f$ satisfies $\Delta f\geq0$, we have $\Delta f=0$.
\end{assumption}
\begin{assumption}\label{assump2}
$(M,g)$ admits an exhaustion function $\rho: M\rightarrow\mathbb{R}_{\geq0}$ with $|\Delta\rho|\in L^1$.
\end{assumption}
\begin{assumption}\label{assump2'}
$(M,g)$ admits an exhaustion function $\rho: M\rightarrow\mathbb{R}_{\geq0}$ with bounded $\Delta\rho$.
\end{assumption}
The above assumptions are introduced for K\"{a}hler manifolds in \cite{Mo2020} and \cite{Si1988}.
\begin{examp}
Compact Riemannian manifolds satisfy Assumptions \ref{assump1}, \ref{assump2} and \ref{assump2'}.
\end{examp}
\begin{examp}\label{exam2}
Let $M$ be a Zariski open subset of a compact K\"{a}hler manifold $(\ol{X},\ol{g})$, $g$ is the restriction of $\ol{g}$, then it satisfies
Assumptions \ref{assump1}, \ref{assump2} and \ref{assump2'}.
\end{examp}
\begin{examp}
Consider $(\mathbb{R}^2,g_{\mathbb{R}^2})$ with the Euclidean metric $g_{\mathbb{R}^2}$ and $\phi_{\mathbb{R}^2}=(1+|r|^2)^{-1-\delta}, \delta>0$, then it has infinite volume and satisfies Assumptions \ref{assump1}. Moreover, $(\mathbb{R}^2,\phi_{\mathbb{R}^2}g_{\mathbb{R}^2})$ satisfies Assumptions \ref{assump1} and has finite volume.
\end{examp}
\begin{examp}
Any product manifold $(\mathbb{R}^2,g_{\mathbb{R}^2})\times(N,g_N)$ has infinite volume and satisfies Assumptions \ref{assump1},
where $(N,g_N)$ is a compact Riemannian manifold.
\end{examp}
\subsection{Stability of vector bundles}
Assume $(E,D)$ be a vector bundle over a Riemannian manifold $(M,g)$, equipped with a metric $K$ on $E$.
\begin{defn}
The $K$-analytically degree of $(E,D)$ is defined by
\begin{equation}\begin{split}
\deg_g(E,D,K)=-\int_M\tr D_K^\ast\psi_K\dvol_g.
\end{split}\end{equation}
For any $D$-invariant sub-bundle $S\subseteq E$, we define
\begin{equation}\begin{split}
\deg_g(S,D,K)=-\int_M\tr D_{K_S}^\ast\psi_{K_S}\dvol_g,
\end{split}\end{equation}
where $K_S$ is the restricted metric on $S$. Then $(E,D)$ is said to be $K$-analytically stable if for any nontrivial $D$-invariant sub-bundle $S$,
it holds that
\begin{equation}\begin{split}
\frac{\deg_g(S,D,K)}{\rank(S)}<\frac{\deg_g(E,D,K)}{\rank(E)}.
\end{split}\end{equation}
\end{defn}
\begin{rem}\label{kahlerstability}
If $(M,\omega)$ be a K\"{a}hler manifold and $D$ is a flat connection, we have
\begin{equation}\begin{split}
\deg_g(E,D,K)=\deg_\omega(E,D^{0,1},K)=-\deg_\omega(E,\delta_K^{0,1},K),
\end{split}\end{equation}
where $\delta_K^{0,1}=D_K^{0,1}-\psi_K^{0,1}$, $\deg_\omega(E,D^{0,1},K)$, $\deg_\omega(E,\delta_K^{0,1},K)$ are the $K$-analytically degree
of the holomorphic bundles $(E,D^{0,1})$ and $(E,\delta_K^{0,1})$ respectively.
\end{rem}
For any $D$-invariant sub-bundle $S$, we have the Chern-Weil formula written in terms of $D_K^\ast\psi_K$,
\begin{equation}\begin{split}\label{Chernweilbundle}
\deg_g(S,D,K)=-\int_M\left(\tr(\pi_K\circ D_K^\ast\psi_K)+\frac{1}{2}|D\pi_K|^2_K\right)\dvol_g,
\end{split}\end{equation}
where $\pi_K$ is the orthogonal projection onto $S$.
\begin{proposition}\label{polystable}
If a vector bundle $(E,D)$ admits a Poisson metric $K$, it is an orthogonal direct sum of $K$-analytically stable bundles.
\end{proposition}
\begin{proof}
For any nontrivial $D$-invariant sub-bundle $S\subset E$, the Chern-Weil formula yields
\begin{equation}\begin{split}
\frac{\deg_g(S,D,K)}{\rank(S)}
&=\frac{-\int_M\left(\tr(\pi_K\circ D_K^\ast\psi_K)+\frac{1}{2}|D\pi_K|^2_K\right)\dvol_g}{\rank(S)}
\\&=\frac{-\int_M\left(\frac{\rank(S)}{\rank(E)}\tr D_K^\ast\psi_K+\frac{1}{2}|D\pi_K|^2_K\right)\dvol_g}{\rank(S)}
\\&\leq\frac{\deg_g(E,D,K)}{\rank(E)}.
\end{split}\end{equation}
If the equality holds, we obtain $D\pi_K=0$. Let $S^\perp$ denote the orthogonal complement of $S$ in $E$ with respect to $K$, then it is also a
$D$-invariant sub-bundle such that $(E,D)\cong(S,D_S)\oplus(S^\perp,D_{S^\perp})$, where $D_S$ and $D_{S^\perp}$ are the induced connections.
Moreover, the induced metrics $K_S$ and $K_{S^\perp}$ both are Poisson metrics and  we complete the proof by an easy induction.
\end{proof}
\subsection{Some useful calculations}
We present some important calculations, it is emphasized that we are not working under the local flat frame and hence most of them need not the
assumption that the underlying connection is flat.
\par Given a vector bundle $E$ over a Riemannian manifold $(M,g)$, the exterior differential operator $D:\Omega^{p}(E)\rightarrow\Omega^{p+1}(E)$
relative to a connection is given by
\begin{equation}\begin{split}
D\omega(e_0,...,e_p)
&=\sum\limits_{k=0}^p(-1)^{k}D_{e_k}(\omega(e_0,e_1,...,\hat{e_k},...,e_p))
\\&+\sum\limits_{k<l}(-1)^{k+l}\omega([e_k,e_l],e_1,...,\hat{e_k},...,\hat{e_l},...,e_p).
\end{split}\end{equation}
Since the Levi-Civita connection on $TM$ is torsion-free, we also have
\begin{equation}\begin{split}
D\omega(e_0,...,e_p)
&=\sum\limits_{k=0}^p(-1)^{k}(\tilde{\nabla}_{e_k}\omega)(e_0,...,\hat{e_k},...,e_p),
\end{split}\end{equation}
where $\tilde\nabla$ is the induced connection on $\Omega^\ast(E)$.
\par For a metric $K$ on $E$, we define the point-wise inner product
\begin{equation}\begin{split}
K(\omega,\theta)(x)=\sum\limits_{i_1<...<i_p}(\omega(e_{i_1},...,e_{i_p}),\theta(e_{i_1},...,e_{i_p}))_K,
\end{split}\end{equation}
where $\omega,\theta\in\Omega^p(E)$ and $\{e_i\}_{i=1}^{\dim M}$ is the orthogonal unit basis of $T_xM$. Then the decomposition $(\ref{decomposition})$
is achieved by choosing $\psi_K$ such that
\begin{equation}\begin{split}\label{psi}
K(\psi_K(X),Y)=\frac{1}{2}(K(DX,Y)+K(X,DY)-dK(X,Y)),
\end{split}\end{equation}
for any $X,Y\in\Omega^0(E)$. With respect to the Riemannian structures of $E$ and $TM$, the co-differential operator
$D_K^\ast:\Omega^{p}(E)\rightarrow\Omega^{p-1}(E)$ is characterized via the formula
\begin{equation}\begin{split}
\int_MK(D_K\omega,\theta)\dvol_g=\int_MK(\omega,D_K^\ast\theta)\dvol_g,
\end{split}\end{equation}
for any $\omega\in\Omega^{p-1}(E),\theta\in\Omega^{p}(E)$ and either $\omega$ or $\theta$ is compactly supported. Then it holds
\begin{equation}\begin{split}
D_K^\ast\theta(e_1,..,e_{p-1})=-\sum\limits_{i}(\nabla_{K,e_i}\theta)(e_i,e_1,...,e_{p-1}),
\end{split}\end{equation}
where $\nabla_K$ is the connection on $\Omega^\ast(E)$, induced by the Levi-Civita connection $\nabla$ and $D_K$.
\par If $H$ be another metric on $E$, we define the positive endomorphism $h=K^{-1}H$ by setting $H(\cdot,\cdot)=K(h(\cdot),\cdot)$.
A straightforward calculation, using $(\ref{psi})$, shows that
\begin{equation}\label{DHDK}\begin{split}
D_H&=D_K+\frac{1}{2}h^{-1}\delta_Kh
=\frac{1}{2}(D_K+h^{-1}\circ D_K\circ h+\psi_K-h^{-1}\circ\psi_K\circ h),
\end{split}\end{equation}
\begin{equation}\label{psiHpsiK}\begin{split}
\psi_H&=\psi_K-\frac{1}{2}h^{-1}\delta_Kh
=\frac{1}{2}(\psi_K+h^{-1}\circ\psi_K\circ h+D_K-h^{-1}\circ D_K\circ h),
\end{split}\end{equation}
\begin{equation}\label{deltaHdeltaK}\begin{split}
\delta_H&=\delta_K+h^{-1}\delta_Kh
=h^{-1}\circ\delta_K\circ h,
\end{split}\end{equation}
where $\delta_K=D_K-\psi_K$.
\begin{lemma}
If $H(t)$ is a one-parameter family of fiber metrics on the vector bundle $(E,D)$, then we have
\begin{equation}\begin{split}\label{DHt}
\frac{\pa D_{H(t)}}{\pa t}=\frac{1}{2}D_{H(t)}(h^{-1}(t)\frac{\pa h(t)}{\pa t})-\frac{1}{2}[\psi_H(t),h^{-1}(t)\frac{\pa h(t)}{\pa t}],
\end{split}\end{equation}
\begin{equation}\begin{split}\label{psiHt}
\frac{\pa\psi_{H(t)}}{\pa t}=-\frac{1}{2}D_{H(t)}(h^{-1}(t)\frac{\pa h(t)}{\pa t})+\frac{1}{2}[\psi_H(t),h^{-1}(t)\frac{\pa h(t)}{\pa t}],
\end{split}\end{equation}
\begin{equation}\begin{split}\label{DHstarpsiHt}
\frac{\pa D_{H(t)}^\ast\psi_{H(t)}}{\pa t}
&=-\frac{1}{2}D_{H(t)}^\ast D_{H(t)}(h^{-1}(t)\frac{\pa h(t)}{\pa t})+\frac{1}{2}[D_{H(t)}^\ast\psi_{H(t)},h^{-1}(t)\frac{\pa h(t)}{\pa t}]
\\&-\frac{1}{2}g^{ij}[\psi_{H(t)}(\frac{\pa}{\pa x^i}),[\psi_{H(t)}(\frac{\pa}{\pa x^j}),h^{-1}(t)\frac{\pa h(t)}{\pa t}]],
\end{split}\end{equation}
where $h(t)=K^{-1}H(t)$.
\end{lemma}
\begin{proof}
It suffice to prove the $(\ref{DHt})$ and $(\ref{DHstarpsiHt})$. By making use of $(\ref{DHDK})$, $(\ref{psiHpsiK})$, we deduce
\begin{equation}\begin{split}
\frac{\pa D_H}{\pa t}&=\frac{1}{2}\frac{\pa}{\pa t}(h^{-1}D_Kh-h^{-1}\circ\psi_K\circ h)
\\&=-\frac{1}{2}h^{-1}\frac{\pa h}{\pa t}h^{-1}D_Kh+\frac{1}{2}h^{-1}D_K\frac{\pa h}{\pa t}+\frac{1}{2}h^{-1}\frac{\pa h}{\pa t}h^{-1}\psi_Kh
-\frac{1}{2}h^{-1}\psi_K\frac{\pa h}{\pa t}
\\&=-\frac{1}{2}h^{-1}\frac{\pa h}{\pa t}h^{-1}D_Kh+\frac{1}{2}h^{-1}D_K\frac{\pa h}{\pa t}-\frac{1}{2}[\psi_H,h^{-1}\frac{\pa h}{\pa t}]
\\&+\frac{1}{2}h^{-1}\frac{\pa h}{\pa t}(\psi_H-\psi_K-D_K+h^{-1}\circ D_K\circ h)
\\&-\frac{1}{2}(\psi_H-\psi_K-D_K+h^{-1}\circ D_K\circ h)h^{-1}\frac{\pa h}{\pa t}
\\&=\frac{1}{2}D_H(h^{-1}\frac{\pa h}{\pa t})-\frac{1}{2}[\psi_H,h^{-1}\frac{\pa h}{\pa t}]
-\frac{1}{2}h^{-1}\frac{\pa h}{\pa t}h^{-1}D_Kh
\\&+\frac{1}{2}h^{-1}D_K\frac{\pa h}{\pa t}+\frac{1}{2}[h^{-1}\frac{\pa h}{\pa t},h^{-1}\circ D_K\circ h]
\\&=\frac{1}{2}D_H(h^{-1}\frac{\pa h}{\pa t})-\frac{1}{2}[\psi_H,h^{-1}\frac{\pa h}{\pa t}].
\end{split}\end{equation}
\par On the other hand, let's choose the local normal coordinate $(x^1,...,x^n)$ at the considered point, using $(\ref{DHt})$ and $(\ref{psiHt})$, we have
\begin{equation}\begin{split}
\frac{\pa D_{H}^\ast\psi_{H}}{\pa t}
&=-\frac{\pa}{\pa t}[D_{H,\frac{\pa}{\pa x^i}},\psi_H(\frac{\pa}{\pa x^i})]
\\&=-[\frac{\pa}{\pa t}D_{H,\frac{\pa}{\pa x^i}},\psi_H(\frac{\pa}{\pa x^i})]-[D_{H,\frac{\pa}{\pa x^i}},\frac{\pa}{\pa t}\psi_H(\frac{\pa}{\pa x^i})]
\\&=-\frac{1}{2}[D_{H,\frac{\pa}{\pa x^i}}(h^{-1}\frac{\pa h}{\pa t})-[\psi_H(\frac{\pa}{\pa x^i}),h^{-1}\frac{\pa h}{\pa t}],\psi_H(\frac{\pa}{\pa x^i})]
\\&+\frac{1}{2}[D_{H,\frac{\pa}{\pa x^i}},D_{H,\frac{\pa}{\pa x^i}}(h^{-1}\frac{\pa h}{\pa t})-[\psi_H(\frac{\pa}{\pa x^i}),h^{-1}\frac{\pa h}{\pa t}]]
\\&=\frac{1}{2}D_{H,\frac{\pa}{\pa x^i}}D_{H,\frac{\pa}{\pa x^i}}(h^{-1}\frac{\pa h}{\pa t})-\frac{1}{2}[\psi_H(\frac{\pa}{\pa x^i}),
[\psi_H(\frac{\pa}{\pa x^i}),h^{-1}\frac{\pa h}{\pa t}]]
\\&-\frac{1}{2}[D_{H,\frac{\pa}{\pa x^i}}(h^{-1}\frac{\pa h}{\pa t}),\psi_H(\frac{\pa}{\pa x^i})]-\frac{1}{2}[D_{H,\frac{\pa}{\pa x^i}},
[\psi_H(\frac{\pa}{\pa x^i}),h^{-1}\frac{\pa h}{\pa t}]]
\\&=-\frac{1}{2}D_H^\ast D_H(h^{-1}\frac{\pa h}{\pa t})+\frac{1}{2}[D_{H}^\ast\psi_{H},h^{-1}\frac{\pa h}{\pa t}]
\\&-\frac{1}{2}[\psi_H(\frac{\pa}{\pa x^i}),[\psi_H(\frac{\pa}{\pa x^i}),h^{-1}\frac{\pa h}{\pa t}]].
\end{split}\end{equation}
\end{proof}
\par Under the local coordinate $(x^1,...,x^n)$, by using $(\ref{DHDK})$ and $(\ref{psiHpsiK})$, it follows
\begin{equation}\begin{split}\label{DHKstarpsiHK}
D_H^\ast\psi_H&=D_K^\ast\psi_K-\frac{1}{2}D_K^\ast(h^{-1}\delta_Kh)-\frac{1}{2}g^{ij}[h^{-1}\delta_{K,\frac{\pa}{\pa x^i}}h,\psi_K(\frac{\pa}{\pa x^j})]
\\&=D_K^\ast\psi_K+\frac{1}{2}g^{ij}D_{\frac{\pa}{\pa x^i}}(h^{-1}\delta_{K,\frac{\pa}{\pa x^j}}h)-\frac{1}{2}g^{ij}
\Gamma_{ij}^kh^{-1}\delta_{K,\frac{\pa}{\pa x^k}}h.
\end{split}\end{equation}
\par From now on, the norms below are taken with respect to $K$ unless indicated explicitly. Based on $(\ref{DHKstarpsiHK})$,
we prove the following lemma.
\begin{lemma}\label{keylemma}
Let $H$ and $K$ be two fiber metrics on the vector bundle $(E,D)$, it holds
\begin{equation}\label{keylemma1}\begin{split}
(D_H^\ast\psi_H-D_K^\ast\psi_K,h)=\frac{1}{2}\Delta\tr h-\frac{1}{2}|h^{-\frac{1}{2}}\delta_Kh|^2,
\end{split}\end{equation}
\begin{equation}\label{keylemma2}\begin{split}
(D_K^\ast\psi_K-D_H^\ast\psi_H,h^{-1})_H=\frac{1}{2}\Delta\tr h^{-1}-\frac{1}{2}|h^{\frac{1}{2}}\delta_Hh^{-1}|_H^2,
\end{split}\end{equation}
where $h=K^{-1}H$.
\end{lemma}
\begin{proof}
We only proof $(\ref{keylemma1})$ and $(\ref{keylemma2})$ follows from $(\ref{keylemma1})$. Choose the local normal
coordinate $(x^1,...,x^n)$ at the considered point, it follows
\begin{equation}\begin{split}
(D_H^\ast\psi_H-D_K^\ast\psi_K,h)
&=-\frac{1}{2}(D_K^\ast(h^{-1}\delta_Kh)+[h^{-1}\delta_{K,\frac{\pa}{\pa x^i}}h,\psi_K(\frac{\pa}{\pa x^i})],h)
\\&=\frac{1}{2}(D_{K,\frac{\pa}{\pa x^i}}(h^{-1}\delta_{K,\frac{\pa}{\pa x^i}}h),h)-\frac{1}{2}([h^{-1}\delta_{K,\frac{\pa}{\pa x^i}}h,
\psi_K(\frac{\pa}{\pa x^i})],h)
\\&=\frac{1}{2}\frac{\pa}{\pa x^i}(h^{-1}\delta_{K,\frac{\pa}{\pa x^i}}h,h)-\frac{1}{2}(h^{-1}\delta_{K,\frac{\pa}{\pa x^i}}h,
D_{K,\frac{\pa}{\pa x^i}}h)
\\&+\frac{1}{2}(h^{-1}\delta_{K,\frac{\pa}{\pa x^i}}h,[\psi_K(\frac{\pa}{\pa x^i}),h])
\\&=\frac{1}{2}\frac{\pa}{\pa x^i}\tr\delta_{K,\frac{\pa}{\pa x^i}}h-\frac{1}{2}(h^{-1}\delta_{K,\frac{\pa}{\pa x^i}}h,
\delta_{K,\frac{\pa}{\pa x^i}}h)
\\&=\frac{1}{2}\Delta\tr h-\frac{1}{2}|h^{-\frac{1}{2}}\delta_Kh|^2.
\end{split}\end{equation}
\end{proof}
Following \cite{Do1985}, the Donaldson's distance on the space of metrics is defined by
\begin{equation}\begin{split}
\sigma(K,H)=\tr(K^{-1}H)+\tr(H^{-1}K)-2\rank(E).
\end{split}\end{equation}
It is obvious that $\sigma(H,K)\geq0$ with equality if and only if $H=K$.
\par By Lemma \ref{keylemma}, we get
\begin{corollary}\label{laplacedistance}Let $H$ and $K$ be two harmonic metrics on the vector bundle $(E,D)$, then we have
\begin{equation}\begin{split}
\Delta\sigma(H,K)=|h^{-\frac{1}{2}}\delta_Kh|^2+|h^{\frac{1}{2}}\delta_Hh^{-1}|_H^2,
\end{split}\end{equation}
where $h=K^{-1}H$. In particular, if $(E,D)$ is simple and $(M,g)$ satisfies the Assumption \ref{assump1}, the mutually bounded harmonic metrics are unique up to scaling.
\end{corollary}
Now we consider the evolution equation
\begin{equation}\begin{split}\label{outlineflow}
H^{-1}(t)\frac{\pa H(t)}{\pa t}=2D^\ast_{H(t)}\psi_{H(t)},
\end{split}\end{equation}
a direct calculation yields
\begin{lemma}\label{distancet}
Let $(E,D)$ be a vector bundle over $(M,g)$, assume $H(t)$ and $K(t)$ be two solutions of the heat flow $(\ref{outlineflow})$, then we have
\begin{equation}\begin{split}
(\frac{\pa}{\pa t}-\Delta)\sigma(H(t),K(t))
=-|\tilde{h}(t)^{-\frac{1}{2}}\delta_{K(t)}\tilde{h}(t)|^2_{K(t)}
-|\tilde{h}^{\frac{1}{2}}(t)\delta_{H(t)}\tilde{h}^{-1}(t)|^2_{H(t)},
\end{split}\end{equation}
where $\tilde{h}(t)=K^{-1}(t)H(t)$.
\end{lemma}
\begin{lemma}Let $(E,D)$ be a vector bundle over $(M,g)$, along the heat flow $(\ref{outlineflow})$ it holds
\begin{equation}\begin{split}
\frac{d}{dt}||\psi_H||_{H,L^2}^2&=-2||D_H^{\ast}\psi_H||_{H,2}^2-2\int_Md\eta_H,
\end{split}\end{equation}
\begin{equation}\begin{split}\label{timemodule}
\frac{d}{dt}||D_H^{\ast}\psi_H||_{H,L^2}^2
=-2||D_HD_H^{\ast}\psi_H||_{H,L^2}^2-2||[\psi_H,D_H^\ast\psi_H]||_{H,L^2}^2+2\int_Md\tilde{\eta}_H,
\end{split}\end{equation}
\begin{equation}\begin{split}\label{modulet}
(\frac{\pa}{\pa t}-\Delta)|D_H^{\ast}\psi_H|_H^2&=-2|D_HD_H^\ast\psi_H|_H^2
-2|[\psi_H,D_H^\ast\psi_H]|_H^2,
\end{split}\end{equation}
where $\eta_H=H(D_H^{\ast}\psi_H,\ast\psi_H)$, $\tilde{\eta}_H=H(D_H^\ast\psi_H,\ast D_HD_H^\ast\psi_H)$. Furthermore, if $D$ is a flat connecton, we have
\begin{equation}\begin{split}\label{phimodulet}
(\frac{\pa}{\pa t}-\Delta)|\psi_H|_H^2
=-|[\psi_H,\psi_H]|^2_H-2(\psi_H\circ\Ric,\psi_H)_H-2|\nabla_H\psi_H|^2_H,
\end{split}\end{equation}
where $\Ric$ is the Ricci transformation of $(M,g)$.
\end{lemma}
\begin{proof}
Firstly, applying the equality $(\ref{psiHt})$ yields
\begin{equation}\begin{split}
\frac{d}{dt}||\psi_H||_{H,L^2}^2
&=2\int_M(-\frac{1}{2}D_H(H^{-1}\frac{\pa H}{\pa t})+\frac{1}{2}[\psi_H,H^{-1}\frac{\pa H}{\pa t}],\psi_H)_{H}\dvol_g
\\&=-2\int_M(D_HD_H^{\ast}\psi_H,\psi_H)_{H}\dvol_g
\\&=-2||D_H^{\ast}\psi_H||_{H,2}^2-2\int_Md\eta_H.
\end{split}\end{equation}
Next, $(\ref{DHstarpsiHt})$ indicates that
\begin{equation}\begin{split}
\frac{\pa D_{H}^\ast\psi_{H}}{\pa t}=-D_H^\ast D_HD_H^\ast\psi_H-g^{ij}[\psi_H(\frac{\pa}{\pa x^i}),[\psi_H(\frac{\pa}{\pa x^j}),D_H^\ast\psi_H]],
\end{split}\end{equation}
as it is easy to check $D_H^\ast\psi_H$ is self-adjoint with respect to $H$, it then follows
\begin{equation}\begin{split}
\frac{d}{dt}||D_H^{\ast}\psi_H||_{H,L^2}^2
&=-2\int_M(D_H^\ast D_HD_H^\ast\psi_H,D_H^\ast\psi_H)_H\dvol_g
\\&-2\int_Mg^{ij}([\psi_H(\frac{\pa}{\pa x^i}),[\psi_H(\frac{\pa}{\pa x^j}),D_H^\ast\psi_H]],D_H^{\ast}\psi_H)_H\dvol_g
\\&=-2||D_HD_H^{\ast}\psi_H||_{H,L^2}^2-2||[\psi_H,D_H^\ast\psi_H]||_{H,L^2}^2+2\int_Md\tilde{\eta}_H,
\end{split}\end{equation}
and
\begin{equation}\begin{split}
(\frac{\pa}{\pa t}-\Delta)|D_H^{\ast}\psi_H|_H^2
&=-2(D_H^\ast D_HD_H^\ast\psi_H,D_H^{\ast}\psi_H)_H
\\&-2g^{ij}([\psi_H(\frac{\pa}{\pa x^i}),[\psi_H(\frac{\pa}{\pa x^j}),D_H^\ast\psi_H]],D_H^{\ast}\psi_H)_H
\\&+2(D_H^\ast D_HD_H^\ast\psi_H,D_H^\ast\psi_H)_H-2|D_HD_H^\ast\psi_H|_H^2
\\&=-2|D_HD_H^{\ast}\psi_H|_H^2-2|[\psi_H,D_H^\ast\psi_H]|_H^2.
\end{split}\end{equation}
\par On the other hand, we have
\begin{equation}\begin{split}
(\frac{\pa}{\pa t}-\Delta)|\psi_H|_H^2
&=2(-\frac{1}{2}D_H(H^{-1}\frac{\pa H}{\pa t})+\frac{1}{2}[\psi_H,H^{-1}\frac{\pa H}{\pa t}],\psi_H)_H
\\&+2(\nabla_H^\ast\nabla_H\psi_H,\psi_H)_H-2|\nabla_H\psi_H|_H^2
\\&=-2(D_HD_H^\ast\psi_H-\nabla_H^\ast\nabla_H\psi_H,\psi_H)_H-2|\nabla_H\psi_H|_H^2,
\end{split}\end{equation}
Now using the flatness of $D$: $D_H\psi_H=0$ and $D_H^2=-\frac{1}{2}[\psi_H,\psi_H]$, it is equal to
\begin{equation}\begin{split}
&2g^{ij}(D_H^2(\cdot,\frac{\pa}{\pa x^i})(\psi_H(\frac{\pa}{\pa x^j}))-\psi_H(R_{\nabla}(\cdot,\frac{\pa}{\pa x^i})(\frac{\pa}{\pa x^j})),\psi_H)_H
-2|\nabla_H\psi_H|_H^2
\\&=-(g^{ij}[[\psi_H,\psi_H(\frac{\pa}{\pa x^i})],\psi_H(\frac{\pa}{\pa x^j})]+2\psi_H\circ\Ric,\psi_H)_H
-2|\nabla_H\psi_H|_H^2
\\&=-|[\psi_H,\psi_H]|^2_H-2(\psi_H\circ\Ric,\psi_H)_H-2|\nabla_H\psi_H|^2_H.
\end{split}\end{equation}
\end{proof}
\section{Analytically stability and existence of Poisson metrics}
Given a vector bundle $(E,D)$ over a compact Riemannian manifold $(X,g)$, consider the following heat flow
\begin{equation}\label{flow1}\begin{split}
\left\{ \begin{array}{ll}
H^{-1}\frac{\pa H}{\pa t}=2D^\ast_{H}\psi_{H},\\
H(0)=K.
\end{array}\right.
\end{split}\end{equation}
And if $X$ has nonempty smooth boundary $\pa X$, for any given compatible data $\tilde{H}$, we impose the boundary condition
$H(t)|_{\pa X}=\tilde{H}$ in the above system.
\par With the calculations in the last sections in hand and following \cite{Do1985}, we have
\begin{proposition}\label{longtimes}
Let $(E,D)$ be a vector bundle over a compact Riemannian manifold $(X,g)$(with possibly nonempty boundary),
then heat flow $(\ref{flow1})$ admits unique solution defined on $[0,\infty)$.
\end{proposition}
\begin{proposition}\label{approximateharmonic}
Assume $(E,D)$ be a vector bundle over a compact Riemannian manifold $(M,g)$, then for any $\epsilon>0$, there exists a
fiber metric $H_\epsilon$ such that $||D_{H_\epsilon}\psi_{H_\epsilon}||_{L^\infty}\leq\epsilon$.
\end{proposition}
\begin{proof}Let $H=H(t)$ be the solution to $(\ref{flow1})$, defined on $[0,+\infty)$. Since $(\ref{flow1})$ is the negative gradient flow
of the energy $||\psi_H||_{H,L^2}$,
by $(\ref{timemodule})$ we know $||D_H^{\ast}\psi_H||_{H,L^2}^2\rightarrow0$, as $t$ goes infinity.
Then the proof can be finished by using $(\ref{modulet})$ and parabolic Moser's iteration.
\end{proof}
\begin{proposition}\label{Dirichlet1}
Let $(E,D)$ be a vector bundle over a compact Riemannian manifold $(X,g)$ with smooth boundary $\pa X$. Then for any given data $K$,
there exists a unique harmonic metric $H$ on $(E,D)$ such that $H=K$ on $\pa X$.
\end{proposition}
\begin{proof}
Due to Corollary \ref{laplacedistance}, it remain to prove the existence statement. By the standard elliptic equation theory,
let $u$ be the solution to the following Dirichlet problem:
\begin{equation}\begin{split}
\Delta u=-|D_K^\ast\psi_K|,u|_{\pa X}=0.
\end{split}\end{equation}
Proposition \ref{longtimes} ensures the existence of long-time solution $H(t)$ of $(\ref{flow1})$ with boundary condition $H(t)|_{\pa X}=K$. Set
\begin{equation}\begin{split}
v(x,t)=\int_0^t|D_{H(s)}^\ast\psi_{H(s)}|_{H(s)}ds-u,
\end{split}\end{equation}
by $(\ref{modulet})$ it is easy to see
\begin{equation}\begin{split}
\left\{ \begin{array}{ll}
(\frac{\pa}{\pa t}-\Delta)v(x,t)\leq0,\\
v(x,0)=-u(x),\\
v(x,t)|_{\pa X}=0.
\end{array}\right.
\end{split}\end{equation}
Then the maximum principle gives
\begin{equation}\begin{split}\label{dirchlet1}
\int_0^t|D_{H(s)}^\ast\psi_{H(s)}|_{H(s)}ds\leq\max\limits_X u.
\end{split}\end{equation}
\par Now for any $t_1\leq t\leq t_2$ and $\tilde{h}(t)=H^{-1}(t_1)H(t),\hat{h}(t)=H^{-1}(t)H(t_1)$, it is straightforward to check
\begin{equation}\begin{split}
&|\frac{\pa}{\pa t}\log\tr\tilde{h}(t)|
\leq2|D_{H(t)}^\ast\psi_{H(t)}|_{H(t)},
|\frac{\pa}{\pa t}\log\tr\hat{h}(t)|
\leq2|D_{H(t)}^\ast\psi_{H(t)}|_{H(t)}.
\end{split}\end{equation}
A simple integration yields
\begin{equation}\begin{split}\label{dirchlet2}
\sigma(H(t_1),H(t_2))\leq2\rank(E)(e^{2\int_{t_1}^{t_2}|D_{H(s)}^\ast\psi_{H(s)}|_{H(s)}ds}-1).
\end{split}\end{equation}
\par So by combining $(\ref{dirchlet1})$ and $(\ref{dirchlet2})$, we know that $H(t)$ converge in $C^0$-topology to $H_\infty$
as $t\rightarrow\infty$. Using $(\ref{dirchlet1})$ and arguing as Lemma 3.3 in \cite{Zh2005}, the elliptic regularity theory implies
that there exists a subsequence $H(t_i)$ converging to $H_\infty$ in $C^\infty$-topology and $H_\infty$ is a harmonic metric satisfying
the desired boundary condition.
\end{proof}
\begin{proposition}\label{Dirichlet2}
Let $(E,D)$ be a vector bundle over a compact Riemannian manifold $(X,g)$ with smooth boundary $\pa X$. Then for any given data $K$,
there exists a unique Poisson metric $H$ on $(E,D)$ such that $H=K$ on $\pa X$ and $\det H=\det K$.
\end{proposition}
\begin{proof}
Due to Proposition \ref{Dirichlet1}, there exists a unique harmonic metric $\tilde{H}$ such that $\tilde{H}=K$ on $\pa X$. Set $H=\tilde{H}e^f$,
it follows
\begin{equation}\begin{split}
D_{H}^\ast\psi_{H}&=-\frac{1}{2}D_{\tilde{H}}^\ast(df\otimes\id_E)=\frac{\tr D_H^\ast\psi_H}{\rank(E)}\id_E.
\end{split}\end{equation}
Now we set $f=\frac{\log\det(\tilde{H}^{-1}K)}{\rank(E)}$, then we have $H=K$ on $\pa X$ and $\det H=\det K$.
\end{proof}
On each $M_s$, the following Dirichlet problem is solvable,
\begin{equation}\begin{split}
\left\{ \begin{array}{ll}
D_{H_s}^\ast\psi_{H_s}=c_s\id_E,\\
H|_{\pa M_s}=K,\\
\det h_s=1,h_s=K^{-1}H_s.
\end{array}\right.
\end{split}\end{equation}
Due to the elementary inequality $\log(\frac{\tr h_s}{\rank(E)})\geq\log(\det h_s)^{\frac{1}{r}}=0$ and the boundary condition, we have
\begin{equation}\begin{split}
\frac{\pa\log(\frac{\tr h_s}{\rank(E)})}{\pa\vec{n}}\geq0,
\end{split}\end{equation}
where $\vec{n}$ is the inward unit normal vector field at $\pa M_s$. We extend $\log(\frac{\tr h_s}{\rank(E)})$ to the function $\widetilde{\log(\frac{\tr h_s}{\rank(E)})}$ on whole $M$ by setting $0$ outside $M_s$.
For any non-negative compactly supported function $\vp$, the Green's formula yields
\begin{equation}\begin{split}
\int_M\widetilde{\log(\frac{\tr h_s}{\rank(E)})}\Delta\vp\dvol_g
&=\int_{M_s}\log(\frac{\tr h_s}{\rank(E)})\Delta\vp\dvol_g
\\&=\int_{M_s}\vp\Delta\log(\frac{\tr h_s}{\rank(E)})\dvol_g
\\&+\int_{\pa M_s}\vp\frac{\pa\log(\frac{\tr h_s}{\rank(E)})}{\pa\vec{n}}\dvol_{\pa M_s}
\\&-\int_{\pa M_s}\log(\frac{\tr h_s}{\rank(E)})\frac{\pa\vp}{\pa\vec{n}}\dvol_{\pa M_s}
\\&\geq\int_{M_s}\vp\Delta\log(\frac{\tr h_s}{\rank(E)})\dvol_g.
\end{split}\end{equation}
Now by Lemma \ref{keylemma}, we see
\begin{equation}\begin{split}
\Delta\log(\frac{\tr h_s}{\rank(E)})
&=\frac{2(D_{H_s}^\ast\psi_{H_s}-D_K^\ast\psi_K,h_s)+|h_s^{-\frac{1}{2}}\delta_Kh_s|^2}{\tr h_s}-\frac{|\tr\delta_Kh_s|^2}{(\tr h_s)^2}
\\&\geq\frac{2(D_{H_s}^\ast\psi_{H_s}-D_K^\ast\psi_K,h_s)}{\tr h_s}
\\&=\frac{-2((D_K^\ast\psi_K)^{\perp},h_s)}{\tr h_s}
+\frac{2((\tr D_{H_s}^\ast\psi_{H_s}-\tr D_K^\ast\psi_K)\id_E,h_s)}{\rank(E)\tr h_s}
\\&\geq-2|(D_K^\ast\psi_K)^{\perp}|
+\frac{\Delta\log\det h_s}{\rank(E)}
\\&\geq-2C\phi,
\end{split}\end{equation}
where we have used the following simple inequality
\begin{equation}\begin{split}
\frac{|\tr\delta_Kh_s|^2}{\tr h_s}\leq|h_s^{-\frac{1}{2}}\delta_Kh_s|^2.
\end{split}\end{equation}
And therefore
\begin{equation}\begin{split}
\int_M\widetilde{\log(\frac{\tr h_s}{\rank(E)})}\Delta\vp\dvol_g
&\geq-2C\int_{M_s}\vp\phi\dvol_g
\geq-2C\int_{M}\vp\phi\dvol_g.
\end{split}\end{equation}
That is, $\widetilde{\Delta\log(\frac{\tr h_s}{\rank(E)})}\geq-2C\phi$ as a distribution. So by Assumption \ref{assump1} we get
\begin{equation}\begin{split}
\max\limits_{M_s}\log(\frac{\tr h_s}{\rank(E)})\leq C_1(1+\int_{M_s}\log(\frac{\tr h_s}{\rank(E)})\phi\dvol_g),
\end{split}\end{equation}
where $C_1$ depends on $C$.
Combing this with $\det h_s=1$ yields
\begin{equation}\begin{split}\label{sec31}
\max\limits_{M_s}|\log h_s|\leq C_2(1+\int_{M_s}|\log h_s|\phi\dvol_g),
\end{split}\end{equation}
for a constant $C_2$ depending only on $C_1$ and $\rank(E)$, where we have used that
\begin{equation}\begin{split}
\log(\frac{\tr h_s}{\rank(E)})\leq|\log h_s|\leq\rank(E)^{\frac{3}{2}}\log\tr h_s.
\end{split}\end{equation}
\par Our goal is to show the quantity
\begin{equation}\begin{split}\label{sec32}
\int_{M_s}|\log h_s|\phi\dvol_g
\end{split}\end{equation}
is uniformly bounded, under the assumption that $D$ is $K$-analytically stable.
\par For further consideration, we recall some notations. Given $\rho\in C^\infty(\mathbb{R},\mathbb{R}),
\Theta\in C^\infty(\mathbb{R}\times\mathbb{R},\mathbb{R})$, $\chi\in\Omega^{\bullet}(\End(E))$ and a self-adjoint bundle
endomorphism $\sigma$, we define $\rho[\sigma]$ and $\Theta[\sigma](\chi)$ as follows. For every $x\in M$, we set
\begin{equation}\begin{split}\label{definition}
\rho[\sigma]=\rho(\lambda_\alpha)e^\alpha\otimes e_\alpha, \Theta[\sigma](\chi)
=\Theta(\lambda_\alpha,\lambda_\beta)\chi_\alpha^\beta e^\alpha\otimes e_\beta,
\end{split}\end{equation}
where $\{e_\alpha\}_{\alpha=1}^{\rank(E)}$ is an orthogonal basis with respect to $K$, such that
\begin{equation}\begin{split}
\sigma(e_\alpha)=\lambda_\alpha e_\alpha, \chi=\chi_\alpha^\beta e^\alpha\otimes e_\beta.
\end{split}\end{equation}
\par Next we prove the following proposition.
\begin{proposition}\label{keyproposition}
Let $(E,D)$ be a vector bundle over a compact Riemannian manifold $(X,g)$ with empty boundary,
and $H,K$ be two fiber metrics with $s=\log(K^{-1}H)$ and $s|_{\pa X}=0$. Then we have
\begin{equation}\begin{split}\label{outlinekeyidentity}
\int_X(D_{H}^{\ast}\psi_{H}-D_K^{\ast}\psi_K,s)\dvol_g=-\frac{1}{2}\int_X(\Theta[s](Ds),Ds)\dvol_g,
\end{split}\end{equation}
in the sense of $(\ref{definition})$, where
\begin{equation}\begin{split}\Theta(x,y)=
 \begin{cases}
	\frac{e^{y-x}-1}{y-x},x\neq y,\\
    1,x=y.
 \end{cases}
 \end{split}\end{equation}
\end{proposition}
\begin{proof}
Consider the local normal coordinate at the considered point and for $h=K^{-1}H$, noting that
$(h^{-1}\delta_{K,\frac{\pa}{\pa x^i}}h,s)=\tr(s\delta_{K,\frac{\pa}{\pa x^i}}s)$, through computing, we have
\begin{equation}\begin{split}\label{key01}
(D_H^\ast\psi_H-D_K^\ast\psi_K,s)
&=-\frac{1}{2}(D_K^\ast(h^{-1}\delta_Kh),s)-\frac{1}{2}([h^{-1}\delta_{K,\frac{\pa}{\pa x^i}}h,\psi_K(\frac{\pa}{\pa x^i})],s)
\\&=\frac{1}{2}(D_{K,\frac{\pa}{\pa x^i}}(h^{-1}\delta_{K,\frac{\pa}{\pa x^i}}h),s)+\frac{1}{2}(h^{-1}\delta_{K,\frac{\pa}{\pa x^i}}h,
[\psi_K(\frac{\pa}{\pa x^i}),s])
\\&=\frac{1}{2}\frac{\pa}{\pa x^i}(h^{-1}\delta_{K,\frac{\pa}{\pa x^i}}h,s)-\frac{1}{2}(h^{-1}\delta_{K,\frac{\pa}{\pa x^i}}h,
D_{K,\frac{\pa}{\pa x^i}}s)
\\&+\frac{1}{2}(h^{-1}\delta_{K,\frac{\pa}{\pa x^i}}h,[\psi_K(\frac{\pa}{\pa x^i}),s])
\\&=\frac{1}{2}\frac{\pa}{\pa x^i}\tr(s\delta_{K,\frac{\pa}{\pa x^i}}s)-\frac{1}{2}(h^{-1}\delta_{K,\frac{\pa}{\pa x^i}}h,
\delta_{K,\frac{\pa}{\pa x^i}}s)
\\&=\frac{1}{2}\frac{\pa}{\pa x^i}\tr(s\frac{\pa}{\pa x^i}s)-\frac{1}{2}(h^{-1}\delta_Kh,\delta_Ks).
\end{split}\end{equation}
\par On the other hand, according to Section 7.4 in \cite{LT1995}, there is an open dense subset $W\subset X$ and for every $x\in W$,
we may choose an orthogonal basis $\{e_\alpha\}_{\alpha=1}^{\rank(E)}$ for $E$ with respect to $K$, defined over a neighborhood of $x$, such that
\begin{equation}\begin{split}
h=\sum\limits_{\alpha=1}^{\rank(E)}e^{\lambda_\alpha}e_\alpha\otimes e^\alpha,s=\sum\limits_{\alpha=1}^{\rank(E)}\lambda_\alpha e_\alpha\otimes e^\alpha.
\end{split}\end{equation}
We set $De_\alpha=A_\alpha^\beta e_\beta$, then it follows
\begin{equation}\begin{split}
Dh\circ h^{-1}=\Theta[s](Ds)=\sum\limits_{\alpha=1}^{\rank(E)}d\lambda_\alpha e_\alpha\otimes e^\alpha+\sum\limits_{\alpha\neq\beta}(1-e^{\lambda_\alpha-\lambda_\beta})A_\beta^\alpha e_\alpha\otimes e^\beta,
\end{split}\end{equation}
and therefore
\begin{equation}\begin{split}\label{key02}
(h^{-1}\delta_Kh,\delta_Ks)&=(Ds,Dh\circ h^{-1})=(\Theta[s](Ds),Ds).
\end{split}\end{equation}
\par Finally, by $(\ref{key01})$ and $(\ref{key02})$, we have
\begin{equation}\begin{split}
\int_X(D_{H}^{\ast}\psi_{H}-D_K^{\ast}\psi_K,s)\dvol_g
&=\frac{1}{4}\int_Xd(\ast d|s|^2)-\frac{1}{2}\int_X(\Theta[s](Ds),Ds)\dvol_g
\\&=-\frac{1}{2}\int_X(\Theta[s](Ds),Ds)\dvol_g.
\end{split}\end{equation}
\end{proof}
\par We shall prove $(\ref{sec32})$ by contradictory argument and the method is a combination of Proposition \ref{keyproposition} and Simpson's trick.
Otherwise, we may assume that there exists a sequence $s\rightarrow\infty$, such that
\begin{equation}\begin{split}
l_s=\int_{M_s}|\log h_s|\phi\dvol_g\rightarrow+\infty.
\end{split}\end{equation}
Let's set $u_s=\log h_s/l_s$ with $\tr u_s=0$ and $\int_{M_s}|u_s|\phi\dvol_g=1$. By $(\ref{sec31})$ we see
\begin{equation}\begin{split}\label{sec33}
\sup\limits_{M_s}|u_s|\leq\frac{C_2}{l_s}(1+l_s)\leq C_3,
\end{split}\end{equation}
for a constant $C_3$ doesn't depend on $s$. Using Proposition \ref{keyproposition} and the fact that $H_s$ is a Poisson metric, one has
\begin{equation}\begin{split}
\int_{M_s}l_s(\Theta[l_su_s](Du_s),Du_s)\dvol_g
&=2\int_{M_s}(D_K^{\ast}\psi_K-D_{H_s}^{\ast}\psi_{H_s},u_s)\dvol_g
\\&=2\int_{M_s}((D_K^\ast\psi_K)^{\perp},u_s)\dvol_g.
\end{split}\end{equation}
Notice that
\begin{equation}
\begin{split}
l\Theta(lx,ly)\rightarrow
\begin{cases}
\frac{1}{x-y},x>y,\\+\infty,x\leq y,
\end{cases}
\end{split}\end{equation}
increases monotonically as $l\rightarrow+\infty$. It holds that
\begin{equation}\label{sec34}\begin{split}
\int_{M_s}(\rho[u_s](Du_s),Du_s)\dvol_g\leq2\int_{M_s}((D_K^\ast\psi_K)^{\perp},u_s)\dvol_g.
\end{split}\end{equation}
for any $\rho:\mathbb{R}\times\mathbb{R}\rightarrow\mathbb{R}$ with $\rho(x,y)<\frac{1}{x-y}$ whenever $x>y$ and $s$ large enough.
\par  Due to the zeroth order estimate of $u_s$, and $|(D_K^\ast\psi_K)^{\perp}|\leq C\phi$, by taking $\rho$ small enough, it follows that $Du_s$
is bounded in $L^2$-norm on any compact subset of $M$. We may assume that $u_s\rightarrow u_\infty$ weakly in $L_{1,\loc}^2$-topology with
$||u_\infty||_{L^\infty}\leq C_3$. Since $\phi\in L^1$, for any $\epsilon>0$ there exists a $s_\epsilon$ such that for any $s\geq s_\epsilon$, it holds
\begin{equation}\label{sec35}\begin{split}
\int_{M-M_s}\phi\dvol_g\leq\epsilon.
\end{split}\end{equation}
Hence for $s_\epsilon\leq s_1\leq s_2$, $(\ref{sec33})$ indicates
\begin{equation}\begin{split}
1-C_3\epsilon
&\leq\int_{M_{s_2}}|u_{s_2}|\phi\dvol_g-\int_{M_{s_2}-M_{s_1}}|u_{s_2}|\phi\dvol_g
=\int_{M_{s_1}}|u_{s_2}|\phi\dvol_g
\leq1.
\end{split}\end{equation}
Fixing $s_1$ and by the weak compactness of $\{u_s\}$ in $L_{1,\loc}^2$-topology, we get
\begin{equation}\begin{split}
1-C_3\epsilon\leq\int_{M_{s_1}}|u_\infty|\phi\dvol_g\leq1,
\end{split}\end{equation}
and then taking $s_1\rightarrow\infty$ and $\epsilon\rightarrow0$ yields $\int_M|u_\infty|\phi\dvol_g=1$, which implies $u_\infty\neq0$.
\par Now it follows from $(\ref{sec33}),(\ref{sec34}),(\ref{sec35})$ and $|(D_K^\ast\psi_K)^{\perp}|\leq C\phi$ that
\begin{equation}\begin{split}
\int_{M_{s_1}}(\rho[u_{s_2}](Du_{s_2}),Du_{s_2})\dvol_g
&\leq\int_{M_{s_2}}(\rho[u_{s_2}](Du_{s_2}),Du_{s_2})\dvol_g
\\&\leq2(\int_{M_{s_2}-M_{s_1}}+\int_{M_{s_1}})((D_K^\ast\psi_K)^{\perp},u_{s_2})\dvol_g
\\&\leq2CC_3\epsilon+2\int_{M_{s_1}}((D_K^\ast\psi_K)^{\perp},u_{s_2})\dvol_g.
\end{split}\end{equation}
Let's take $s_2\rightarrow\infty, s_1\rightarrow\infty$ and $\epsilon\rightarrow0$ successively, it yields
\begin{equation}\label{sec36}\begin{split}
\int_M(\rho[u_\infty](Du_\infty),Du_\infty)\dvol_g
&\leq2\int_M((D_K^\ast\psi_K)^{\perp},u_\infty)\dvol_g
\leq2CC_3\int_M\phi\dvol_g.
\end{split}\end{equation}
\par As long as $(\ref{sec36})$ is established, by using Simpson's discussion in \cite{Si1988}(p886-888) and Loftin's observation in \cite{Lo2009},
it's easy to conclude $(\ref{sec32})$ and therefore we obtain
\begin{proposition}\label{poissonc0}
If $(M,g)$ satisfies the Assumption \ref{assump1} and the flat bundle $(E,D)$ is $K$-analytically stable with $|(D_K^\ast\psi_K)^{\perp}|\leq C\phi$ for a constant $C$, it holds
\begin{equation}\begin{split}
\max\limits_{M_s}|\log h_s|\leq\tilde{C}_1,
\end{split}\end{equation}
for a constant $\tilde{C}_1$ that doesn't depend on $s$.
\end{proposition}
Using Lemma \ref{keylemma} again shows
\begin{equation}\begin{split}
|h_s^{-\frac{1}{2}}\delta_Kh_s||_{L^2(M_s)}^2
&=\int_{M_s}\left(\Delta \tr h_s-2(D_{H_s}^\ast\psi_{H_s}-D_K^\ast\psi_K,h_s)\right)\dvol_g
\\&=\int_{M_s}\left(\Delta \tr h_s+2((D_K^\ast\psi_K)^{\perp},h_s)\right)\dvol_g
\\&=-\int_{\pa M_s}\frac{\pa\tr h_s}{\pa\vec{n}}\dvol_{\pa M_s}+2\int_{M_s}((D_K^\ast\psi_K)^{\perp},h_s)\dvol_g
\\&\leq2\int_{M_s}((D_K^\ast\psi_K)^{\perp},h_s)\dvol_g
\\&\leq CC_4\int_M\phi\dvol_g,
\end{split}\end{equation}
where $C_4$ depends only on $\tilde{C}_1$ and $\rank(E)$. So we obtain
\begin{proposition}\label{poissonl2}
If $(M,g)$ satisfies the Assumption \ref{assump1} and the flat bundle $(E,D)$ is $K$-analytically stable with $|(D_K^\ast\psi_K)^{\perp}|\leq C\phi$ for a constant $C$, we have
\begin{equation}\begin{split}||Dh_s||_{L^2(M_s)}\leq\tilde{C_2}||\phi||_{L^1}^{\frac{1}{2}},
\end{split}\end{equation}
for a constant $\tilde{C}_2$ that depends only on $C,\tilde{C}_1$ and $\rank(E)$.
\end{proposition}
\begin{proposition}\label{poissonexistence}
If $(M,g)$ satisfies the Assumption \ref{assump1} and the flat bundle $(E,D)$ is $K$-analytically stable with $|(D_K^\ast\psi_K)^{\perp}|\leq C\phi$ for a constant $C$, then there exists a Poisson metric $H$ such that $\det h=1$, $|h|\in L^\infty$ and $|Dh|\in L^2$ for $h=K^{-1}H$.
\end{proposition}
\begin{proof}
Using the flatness of $D$, we deduce
\begin{equation}\label{sec37}\begin{split}
\Delta|\psi_{H_s}|^2_{H_s}&\geq-2g^{ij}(D_{H_s}^2(\cdot,\frac{\pa}{\pa x^i})(\psi_{H_s}(\frac{\pa}{\pa x^j})),\psi_{H_s})_{H_s}
\\&+2g^{ij}(\psi_{H_s}(R_{\nabla}(\cdot,\frac{\pa}{\pa x^i})(\frac{\pa}{\pa x^j})),\psi_{H_s})_{H_s}
-2(D_{H_s}D_{H_s}^\ast\psi_{H_s},\psi_{H_s})_{H_s}
\\&=|[\psi_{H_s},\psi_{H_s}]|_{H_s}^2+2(\psi_{H_s}\circ\Ric,\psi_{H_s})-2(D_{H_s}D_{H_s}^\ast\psi_{H_s},\psi_{H_s})_{H_s}
\\&\geq-C_5|\psi_{H_s}|_{H_s}^2-\frac{2}{\rank(E)}(d\tr D_{H_s}^\ast\psi_{H_s}\otimes\id_E,\psi_{H_s})_{H_s}
\\&=-C_5|\psi_{H_s}|_{H_s}^2-\frac{2}{\rank(E)}(d\tr D_K^\ast\psi_K\otimes\id_E,\psi_{H_s})_{H_s}
\\&\geq-C_6|\psi_{H_s}|_{H_s}^2-C_7,
\end{split}\end{equation}
where $C_6$ depends only on the lower bounded of the Ricci curvature at the considered point, $C_7$ depends only on $|d\tr D_K^\ast\psi_K|$
and $\rank(E)$.
\par On the other hand, $(\ref{psiHpsiK})$ implies there exists two constants $C_8,C_9$ that depend only on $\tilde{C}_1$
and $\rank(E)$, such that
\begin{equation}\label{sec38}\begin{split}
|\psi_{H_s}|_{H_s}^2\leq C_8(|\psi_K|^2+|Dh_s|^2),
\end{split}\end{equation}
\begin{equation}\label{sec39}\begin{split}
|Dh_s|^2\leq C_9(|\psi_K|^2+|\psi_{H_s}|^2_{H_s}).
\end{split}\end{equation}
So combing $(\ref{sec38})$, $(\ref{sec39})$ with Proposition \ref{poissonl2}, $(\ref{sec37})$ and the Theorem 9.20 in \cite{GT2001},
we get the uniform local boundedness of $|Dh_s|$. Now on each $M_s$, by the Poisson metric equation, we have
\begin{equation}\begin{split}
D_K^\ast(h_s^{-1}\delta_Kh_s)=2(D_K^\ast\psi_K)^\perp-g^{ij}[h^{-1}_s\delta_{K,\frac{\pa}{\pa x^i}}h_s,\psi_K(\frac{\pa}{\pa x^j})],
\end{split}\end{equation}
the standard bootstrapping procedure implies the uniform local higher order estimates of $\{h_s\}$ and by diagonal subsequence argument,
$H_s$ converge to a metric $H_\infty$ on the whole $M$ in $C_{\loc}^\infty$-topology. It obvious that $H_\infty$ is a Poisson metric on
$(E,D)$ such that $\det h=1$, $|h|\in L^\infty$ and $|Dh|\in L^2$.
\end{proof}
\begin{proof}[\textup{\textbf{Proof of Theorem \ref{thm1}}}]
The existence is given by Proposition \ref{poissonexistence} and the uniqueness will be proved in Proposition \ref{uniqueness1}. Conversely, we assume $D$ is irreducible and $H$ be a Poisson metric which shares the same properties with $H_\infty$, we shall demonstrate that it is also analytically stable with respect to $K$. For any $D$-invariant sub-bundle $S\subset E$, $K$ and $H$ restrict to the metrics $K_S$ and $H_S$ on $S$.
Taking the orthogonal complement of $S$ in $E$ with respect to $K$, we have the orthogonal decomposition $E=S\oplus S^\perp$ and the
projection $\pi_K$ onto $S$.
\par If $|D\pi_K|\not\in L^2$, by $(\ref{Chernweilbundle})$ and $|D_K^\ast\psi_K|\in L^1$, we know
\begin{equation}\begin{split}
\deg_g(S,D,K)=-\infty.
\end{split}\end{equation}
\par If $|D\pi_K|\in L^2$, we know $\deg_g(S,D,K)$ is finite. Moreover, we have
\begin{equation}\begin{split}
h_S=\pi_K\circ h\circ\pi_K,
\end{split}\end{equation}
\begin{equation}\begin{split}
D_Sh_S=\pi_K\circ Dh\circ\pi_K+D\pi_K\circ(\id_E-\pi_K)\circ h\circ\pi_K,
\end{split}\end{equation}
where $h_S=K_S^{-1}H_S$ and $D_S$ is the induced connection on $S$. It then follows
\begin{equation}\begin{split}
\int_M|\tr(h_S^{-1}\delta_{D_S,K_S}h_S)|^2\dvol_g
&\leq\int_M|h_S^{-1}|^2|D_Sh_S|^2\dvol_g<\infty,
\end{split}\end{equation}
where $\delta_{D_S,K_S}=D_{S,K}-\psi_{D_S,K_S}$. So Simpson's Lemma 5.2 in \cite{Si1988}(or Yau's Lemma in \cite{Ya1976}) implies
\begin{equation}\begin{split}
\lim\limits_{j\rightarrow\infty}\int_{M_j}\frac{1}{2}d(\ast d\log\det h_S)
=\lim\limits_{j\rightarrow\infty}\int_{M_j}\frac{1}{2}d(\ast\tr(h_S^{-1}\delta_{D_S,K_S}h_S))=0,
\end{split}\end{equation}
where each $M_j$ is an exhaustion subset. It then follows from the convergence theorem that
\begin{equation}\begin{split}
\lim\limits_{j\rightarrow\infty}\int_{M_j}\tr D_{H_S}^\ast\psi_{H_S}\dvol_g
&=\lim\limits_{j\rightarrow\infty}\int_{M_j}\tr D_{K_S}^\ast\psi_{K_S}\dvol_g
\\&+\lim\limits_{j\rightarrow\infty}\int_{M_j}\frac{1}{2}\Delta\log\det h_S\dvol_g
\\&=-\deg_g(S,D,K),
\end{split}\end{equation}
and hence
\begin{equation}\begin{split}
\lim\limits_{j\rightarrow\infty}\int_{M_j}\tr D_{H_S}^\ast\psi_{H_S}\dvol_g
&=\int_{M}\tr D_{H_S}^\ast\psi_{H_S}\dvol_g
=-\deg_g(S,D,H).
\end{split}\end{equation}
\par In summary, we conclude either
\begin{equation}\begin{split}\label{sec310}
\deg_g(S,D,K)=-\infty
\end{split}\end{equation}
or
\begin{equation}\begin{split}\label{sec311}
\deg_g(S,D,K)=\deg_g(S,D,H).
\end{split}\end{equation}
On the other hand, it follows from Proposition \ref{polystable} that $E$ is $H$-analytically stable and hence the proof is completed.
\end{proof}
\section{Some applications and consequences}
\subsection{Uniqueness of Poisson metrics}
Firstly, we get the following uniqueness result, which implies the limiting Poisson metric in Theorem \ref{thm1} is unique.
\begin{proposition}\label{uniqueness1}
Let $(E,D)$ be a $K$-analytically stable bundle over a Riemannian manifold $(M,g)$ satisfying the Assumptions \ref{assump1} and either Assumption
\ref{assump2} or complete with finite volume. Suppose $\tr D_K^\ast\psi_K\in L^1$ and $H$ be a Poisson metric such that $\det h=1,|h|\in L^\infty$ and $|Dh|\in L^2$, where $h=K^{-1}H$.
If $\tilde{H}$ is another Poisson metric which is mutually bounded with $H$ and $\det H=\det\tilde{H}$, we have $H=\tilde{H}$.
\end{proposition}
\begin{proof}
Let $\tilde{h}=H^{-1}\tilde{H}$ and then we have
\begin{equation}\begin{split}
\Delta\tr\tilde{h}
&=2(D_{\tilde{H}}^\ast\psi_{\tilde{H}}-D_{H}^\ast\psi_{H},\tilde{h})_H+|\tilde{h}^{-\frac{1}{2}}\delta_{H}\tilde{h}|^2_{H}
\\&=\frac{(\Delta\log\det\tilde{h}\id_E,\tilde{h})_H}{\rank(E)}+|\tilde{h}^{-\frac{1}{2}}\delta_{H}\tilde{h}|^2_{H}
\\&=|\tilde{h}^{-\frac{1}{2}}\delta_{H}\tilde{h}|^2_{H}.
\end{split}\end{equation}
So Assumption \ref{assump1} implies $\tilde{h}$ is $D$-parallel and set $E=\mathop{\bigoplus}\limits_{j=1}^mE_j$ be the
eigendecomposition of $\tilde{h}$, then this decomposition is orthogonal with respect $H$ and $\tilde{H}$.
Moreover, we have $\tilde{H}|_{E_j}=c_jH|_{E_j}$ for some constants $c_j$ and
\begin{equation}\begin{split}
-\int_M\tr D_{H}^\ast\psi_H\dvol_g=-\sum\limits_{j=1}^m\int_M\tr D_{H|_{E_j}}^\ast\psi_{H|_{E_j}}\dvol_g.
\end{split}\end{equation}
\par Then running the same argument in the proof of Theorem \ref{thm1} shows that $D$ being $H$-analytically stable. As a consequence,
\begin{equation}\begin{split}
\frac{\deg(E,D,H)}{\rank(E)}
&=-\sum\limits_{j=1}^m\frac{\int_M\tr D_{H|_{E_j}}^\ast\psi_{H|_{E_j}}\dvol_g}{\rank(E_j)}\frac{\rank(E_j)}{\rank(E)}
\\&\leq\frac{\deg(E,D,H)}{\rank(E)},
\end{split}\end{equation}
the strict inequality occurs if and only if $m>1$ and therefore we complete the proof.
\end{proof}
If $D$ being flat, the following is a flat bundle analogy of Mochizuki's uniqueness result on Hermitian-Einstein metrics in \cite{Mo2020}.
\begin{proposition}\label{uniqueness2}
Let $(E,D)$ be a $K$-analytically stable flat bundle over a Riemannian manifold $(M,g)$ satisfying Assumption \ref{assump1} and there is a
positive exhaustion function $\rho$ such that $|d\log\rho|\in L^1$. Assume that $|\psi_K|\in L^2$ and $H_1$ be a Poisson metric such that
$|h_1|,|h_1^{-1}|\in L^\infty$ and $|Dh_1|\in L^2$, where $h_1=K^{-1}H_1$. If $H_2$ is another Poisson metric which is mutually bounded with
$H_1$ and $\det H_1=\det H_2$, then $H_1=H_2$.
\end{proposition}
\begin{proof}
Similar, set $E=\mathop{\bigoplus}\limits_{j=1}^mE_j$ be the eigendecomposition of $h_{12}=H_1^{-1}H_2$ and $H_1|_{E_j}=c_iH_2|_{E_j}$
for some constants $c_j$. Let $\pi_j$ denotes the projection onto $E_j$ and $h_i=K^{-1}H_i$, then $\pi_j$ are bounded with respect to $K$,
as $K$ and $H_i$ are mutually bounded.
\par Now $(\ref{deltaHdeltaK})$ indicates $\delta_K\pi_j=\delta_{H_1}\pi_j-[h_1^{-1}\delta_Kh_1,\pi_j]=-[h_1^{-1}\delta_Kh_1,\pi_j]$, this implies $\delta_K\pi_j$ and $D\pi_j^{\ast K}$ are $L^2$. We consider the fiber metric $H$ determined by the direct sum of $K|_{E_j}$, then $H$
and $K$ are mutually bounded. By definition, we have $h:=K^{-1}H=\sum\limits_{j=1}^m\pi_j^{\ast K}\circ\pi_j$ and it follows
\begin{equation}\begin{split}
h^{-1}Dh=h^{-1}\sum\limits_{j=1}^mD\pi_j^{\ast K}\circ\pi_j,
h^{-1}\delta_Kh=h^{-1}\sum\limits_{j=1}^m\pi_j^{\ast K}\circ\delta_K\pi_j,
\end{split}\end{equation}
which implies $h^{-1}Dh$ and $h^{-1}\delta_Kh$ are $L^2$.
\par Next we compute $D(h^{-1}\delta_Kh)=-h^{-1}Dhh^{-1}\delta_Kh+h^{-1}D(\delta_Kh)$ and
\begin{equation}\begin{split}
D(\delta_Kh)
&=D(\sum\limits_{j=1}^m\pi_j^{\ast K}\circ\delta_K\pi_j)
\\&=\sum\limits_{j=1}^m\left(D\pi_j^{\ast K}\circ\delta_K\pi_j+\pi_j^{\ast K}\circ D(\delta_K\pi_j)\right)
\\&=\sum\limits_{j=1}^m\left(D\pi_j^{\ast K}\circ\delta_K\pi_j+4\pi_j^{\ast K}\circ [D_K^2,\pi_j]\right)
\\&=\sum\limits_{j=1}^m\left(D\pi_j^{\ast K}\circ\delta_K\pi_j-2\pi_j^{\ast K}\circ [[\psi_K,\psi_K],\pi_j]\right).
\end{split}\end{equation}
where we have use the flatness of $D$. As a consequence, we find $\Delta\log\det h$ is $L^1$.
\par Finally we set $\chi_N=\eta(\frac{\rho}{N})$, where $\eta$ is a nonnegative function with $\eta(t)=0$ if $t\geq2$ and $\eta(t)=1$ if $t\leq1$.
We then deduce
\begin{equation}\begin{split}
\int_M\chi_N\Delta\log\det h\dvol_g
&=-\int_M\nabla\chi_N\cdot\nabla\log\det h\dvol_g
\\&=-\int_{\{N\leq\rho\leq2N\}}\frac{\eta^{'}(\frac{\rho}{N})}{N}\nabla\rho\cdot\nabla\log\det h\dvol_g
\\&\leq\int_{\{N\leq\rho\leq2N\}}\frac{2|\eta^{'}(\frac{\rho}{N})|}{\rho}|\nabla\rho||\nabla\log\det h|\dvol_g
\\&=\int_{\{N\leq\rho\leq2N\}}2|\eta^{'}(\frac{\rho}{N})||\nabla\log\rho||\tr(h^{-1}\delta_Kh)|\dvol_g.
\end{split}\end{equation}
Combing this and using the fact that $\Delta\log\det h$ lies in $L^1$, it follows
\begin{equation}\begin{split}
-\int_M\tr D_{K}^\ast\psi_K\dvol_g&=-\int_M\tr D_H^\ast\psi_H\dvol_g+\frac{1}{2}\int_M\Delta\log\det h\dvol_g
\\&=-\sum\limits_{j=1}^m\int_M\tr D_{K|_{E_j}}^\ast\psi_{K|_{E_j}}\dvol_g.
\end{split}\end{equation}
Since $D$ is $K$-analytically stable, we obtain $m=1$ and the proof is completed.
\end{proof}
\subsection{Applications to the nonableian Hodge correspondence}
\begin{proposition}\label{Higgsstructure1}
Let $(X,\omega)$ be a noncompact complex curve satisfying the Assumption \ref{assump1}, and $(E,D)$ be a $K$-analytically stable flat bundle over $X$
with $|(D_K^\ast\psi_K)^\perp|\leq C\phi$ for a constant $C$, then there exists a Higgs structure on $E$.
\end{proposition}
\begin{proof}
According to Theorem \ref{thm1}, there is a Poisson metric $H$ on $(E,D)$ and we have
\begin{equation}\begin{split}
D_H\psi_H^{\perp}=D_H(\psi_H^{\perp,1,0}+\psi_H^{\perp,0,1})=0,
\end{split}\end{equation}
\begin{equation}\begin{split}
D_H^\ast\psi_H^\perp=\sqrt{-1}\ast D_H(\psi_H^{\perp,1,0}-\psi_H^{\perp,0,1})=0,
\end{split}\end{equation}
where $\psi_H^\perp$ is the trace-free part of $\psi_H$. Then we easily get $D_H^{0,1}\psi_H^{\perp,1,0}=0$ and the pair
$(D_H^{0,1},\psi_H^{\perp,1,0})$ gives rise a Higgs structure on $E$.
\end{proof}
For higher dimensional case, we have
\begin{proposition}\label{Higgsstructure2}
Let $(X,\omega)$ be a complete K\"{a}hler manifold with bounded Ricci curvature from below
and satisfying the Assumption \ref{assump1}. Assume $(E,D)$ be a $K$-analytically stable flat bundle over $X$ with
$|(D_K^\ast\psi_K)^\perp|\leq C\phi$ for a constant $C$ and $|\psi_K|\in L^2$. Then there exists a Higgs structure on $E$.
\end{proposition}
\begin{proof}
Due to Theorem \ref{thm1}, there is a Poisson metric $H$ and we compute
\begin{equation}\begin{split}\label{sec421}
\Delta|\psi_H^\perp|^2_H&=-2(\nabla_H^\ast\nabla_H\psi_H^\perp,\psi_H^\perp)_H+2|\nabla_H\psi_H^\perp|^2_H
\\&=-2g^{ij}(D_H^2(\cdot,\frac{\pa}{\pa x^i})(\psi_H^\perp(\frac{\pa}{\pa x^j})),\psi_H^\perp)_H
\\&+2g^{ij}(\psi_H^\perp(R_{\nabla}(\cdot,\frac{\pa}{\pa x^i})(\frac{\pa}{\pa x^j})),\psi_H^\perp)_H+2|\nabla_H\psi_H^\perp|_H^2
\\&=g^{ij}([[\psi_H,\psi_H(\frac{\pa}{\pa x^i})],\psi_H^\perp(\frac{\pa}{\pa x^j})]+2\psi_H^\perp\circ\Ric,\psi_H^\perp)_H
+2|\nabla_H\psi_H^\perp|_H^2
\\&=|[\psi_H^\perp,\psi_H^\perp]|_H^2+2(\psi_H^\perp\circ\Ric,\psi_H^\perp)_H+2|\nabla_H\psi_H^\perp|^2_H
\\&\geq-C_\ast|\psi_H^\perp|_H^2+2|\nabla_H\psi_H^\perp|^2_H,
\end{split}\end{equation}
where $-C_\ast$ the lower bound of Ricci curvature. Let $R$ be any positive constant and we fix point $x_0$, choose a cut-off function $\eta$ satisfying
\begin{equation}\begin{split}
\left\{ \begin{array}{ll}
\eta(x)=1, x\in B_{x_0}(R),\\
\eta(x)=0,x\in X\setminus B_{x_0}(2R),\\
0\leq\eta\leq1, |\nabla\eta|\leq\frac{C_1}{R},
\end{array}\right.
\end{split}\end{equation}
where $C_1$ is a positive constant and $B_{x_0}(R)$ is the geodesics ball centered at $x_0$ with radius $R$. Then we have
\begin{equation}\begin{split}
2\int_X\eta^2|\nabla_H\psi_H^\perp|^2_H\dvol_\omega
&\leq\int_X\eta^2(\Delta|\psi_H^\perp|^2_H+C_\ast|\psi_H^\perp|^2_H)\dvol_\omega
\\&\leq\int_X4\eta|\nabla\eta||\psi_H^\perp|_H|\nabla_H\psi_H^\perp|_H\dvol_\omega+C_\ast\int_X\eta^2|\psi_H^\perp|_H^2\dvol_\omega
\\&\leq\int_X\eta^2|\nabla_H\psi_H^\perp|_H^2\dvol_\omega+\int_X(C_\ast\eta^2+4|\nabla\eta|^2)|\psi_H^\perp|_H^2\dvol_\omega
\\&\leq\int_X\eta^2|\nabla_H\psi_H^\perp|_H^2\dvol_\omega+\int_X(C_\ast\eta^2+\frac{4C_1^2}{R^2})|\psi_H^\perp|_H^2\dvol_\omega.
\end{split}\end{equation}
And it also holds
\begin{equation}\begin{split}\label{sec422}
\int_{X}|\psi_H^\perp|_H^2\dvol_\omega
&=\int_{X}|\psi_K^\perp-\frac{1}{2}h^{-1}Dh+h^{-1}[\psi_K^\perp,h]|_H^2\dvol_\omega
<\infty,
\end{split}\end{equation}
where $h=K^{-1}H$. So by letting $R$ goes infinity, we see
\begin{equation}\begin{split}\label{sec423}
\int_X|\nabla_H\psi_H^\perp|^2_H\dvol_\omega\leq C_\ast\int_{X}|\psi_H^\perp|_H^2\dvol_\omega<\infty.
\end{split}\end{equation}
\par Now let's consider the pseudo-curvature operator $G_H=(D_H^{0,1}+\psi_H^{1,0})^2$. Notice the flatness of $D$ implies
\begin{equation}\begin{split}\label{sec424}
D_H^{1,0}\psi_H^{0,1}=-D_H^{0,1}\psi_H^{1,0},D_H^{1,0}D_H^{1,0}=-\frac{1}{2}[\psi_H^{1,0},\psi_H^{1,0}],
D_H^{0,1}D_H^{0,1}=-\frac{1}{2}[\psi_H^{0,1},\psi_H^{0,1}],
\end{split}\end{equation}
and it follows that
\begin{equation}\begin{split}
G_H^\ast
&=\left((D_H^{0,1})^2+D_H^{0,1}\psi_H^{1,0}+\psi_H^{1,0}\wedge\psi_H^{1,0}\right)^\ast
\\&=-D_H^{1,0}D_H^{1,0}+D_H^{1,0}\psi_H^{0,1}-\psi_H^{0,1}\wedge\psi_H^{0,1}
\\&=G^{2,0}_H-G_H^{1,1}+G^{0,2}_H.
\end{split}\end{equation}
By K\"{a}hler identity we know $\sqrt{-1}\Lambda_\omega G_H^\perp=\frac{1}{2}(D_H^\ast\psi_H)^\perp=0$ and
\begin{equation}\begin{split}\label{sec425}
\tr(G_H^\perp\wedge G_H^\perp)\wedge\frac{\omega^{n-2}}{(n-2)!}
&=\tr\left(G_H^\perp\wedge\ast(G_H^\perp)^\ast\right)
=|G_H^\perp|_H^2\dvol_\omega,
\end{split}\end{equation}
where $n$ is the complex dimension of $X$. On the other hand, a direct calculation shows
\begin{equation}\begin{split}\label{sec426}
\tr(G_H^\perp\wedge G_H^\perp)\wedge\omega^{n-2}
&=-2\tr(\psi_H^{0,1}\wedge\psi_H^{0,1}\wedge\psi_H^{1,0}\wedge\psi_H^{1,0})\wedge\omega^{n-2}
\\&-\tr(D_H^{0,1}\psi_H^{1,0}\wedge D_H^{1,0}\psi_H^{0,1})\wedge\omega^{n-2}
\\&-\frac{\tr(D_H^{0,1}\psi_H^{1,0})\wedge\tr(D_H^{0,1}\psi_H^{1,0})}{\rank(E)}\wedge\omega^{n-2}
\\&=-\ol\pa\left(\tr(\psi_H^{1,0}\wedge D_H^{1,0}\psi_H^{0,1})\wedge\omega^{n-2}\right)
\\&+\ol\pa\left(\frac{\tr \psi_H^{1,0}\wedge\tr(D_H^{1,0}\psi_H^{0,1})}{\rank(E)}\wedge\omega^{n-2}\right)
\\&=-\ol\pa\left(\tr(\psi_H^{1,0}\wedge D_H^{1,0}\psi_H^\perp)\wedge\omega^{n-2}\right).
\end{split}\end{equation}
where we have used $(\ref{sec424})$, $D_H^{0,1}\psi_H^{0,1}=0$ and $[D_H^{1,0},D_H^{0,1}]=-[\psi_H^{1,0},\psi_H^{0,1}]$.
\par Now as it holds
\begin{equation}\begin{split}
D_H^{1,0}\psi_H^\perp=\sum\limits_{i=1}^ndz^{i}\wedge\nabla_{H,\frac{\pa}{\pa z^i}}\psi_H^\perp,
\end{split}\end{equation} it follows from $(\ref{sec423})$ that
\begin{equation}\begin{split}\label{sec427}
\int_X|\tr(\psi_H^{1,0}\wedge D_H^{1,0}\psi_H^\perp)\wedge\omega^{n-2}|\dvol_\omega<\infty.
\end{split}\end{equation}
Therefore, by $(\ref{sec425})$, $(\ref{sec426})$, $(\ref{sec427})$ and Yau's Lemma in \cite{Ya1976}, we conclude $G_H^\perp$ vanishes
identically and the pair $(D_H^{0,1},\psi_H^{\perp,1,0})$ is a Higgs structure on $E$.
\end{proof}
\begin{proof}[\textup{\textbf{Proof of Theorem \ref{thm2}}}]
Proposition \ref{Higgsstructure1} and Proposition \ref{Higgsstructure2} indicate $(D_H^{0,1},\psi_H^{\perp,1,0})$
determines a Higgs structure on $E$. Moreover, we see the corresponding Hitchin-Simpson curvature
\begin{equation}\begin{split}
F_{D_H^{0,1},\psi_H^{\perp,1,0},H}
=D_H^2+D_H\psi_H^{\perp}+\frac{1}{2}[\psi_H^{\perp},\psi_H^{\perp}]=0,
\end{split}\end{equation}
where we have used the flatness of $D$.
Hence we complete of proof of $(1)$ and $(2)$. For $(3)$, if $(X,\omega)$ has nonnegative Ricci curvature, we have for any $\epsilon>0$ that
\begin{equation}\begin{split}
\Delta(\frac{1}{2}|\psi_H^\perp|^2_H+\epsilon)^{\frac{1}{2}}
&=\frac{1}{4}\frac{\Delta|\psi_H^\perp|^2_H}{(\frac{1}{2}|\psi_H^\perp|^2_H+\epsilon)^{\frac{1}{2}}}
-\frac{1}{4}\frac{|\frac{1}{2}d|\psi_H^\perp|^2_H|^2}{(\frac{1}{2}|\psi_H^\perp|^2_H+\epsilon)^{\frac{3}{2}}}
\\&\geq\frac{1}{4}\frac{2|\nabla_H\psi_H^\perp|^2_H}{(\frac{1}{2}|\psi_H^\perp|^2_H+\epsilon)^{\frac{1}{2}}}
-\frac{1}{4}\frac{|\psi_H^\perp|^2_H|\nabla_H\psi_{H}^\perp|_H^2}{(\frac{1}{2}|\psi_H^\perp|^2_H+\epsilon)^{\frac{3}{2}}}
\\&=\frac{1}{2}\frac{|\nabla_H\psi_{H}^\perp|_H^2}{(\frac{1}{2}|\psi_H^\perp|^2_H+\epsilon)^{\frac{1}{2}}}(1-\frac{\frac{1}{2}|\psi_H^\perp|^2_H}
{\frac{1}{2}|\psi_{H}^\perp|_H^2+\epsilon})
\\&\geq0,
\end{split}\end{equation}
and by letting $\epsilon$ goes zero, it follows that $|\psi_H^\perp|_H$ is subharmonic. On the other hand, as $||\psi_H^\perp||_{H,L^2}<\infty$,
Yau's Liouville theorem in \cite{Ya1976} applies and we conclude $|\psi_H^\perp|_H$ is constant.
But the volume of $(X,\omega)$ is infinite, so we know $\psi_H^\perp=0$ and then $(3)$ follows easily.
\end{proof}
In a similar way, using Corollary \ref{cor2} instead, we obtain the following proposition, which means the harmonic metric produced there is
pluriharmonic.
\begin{proposition}\label{Higgsstructure3}
Let $(X,\omega)$ be a K\"{a}hler manifold satisfying the Assumptions \ref{assump1}, \ref{assump2'} with $\phi=1$ and $(E,D)$ be a
$K$-analytically stable flat bundle over $X$ such that $|D_K^\ast\psi_K|\in L^\infty$ and $|\psi_K|\in L^2$.
Assume either $\dim_{\mathbb{C}}X=1$ or $(X,\omega)$ being complete with bounded Ricci curvature from below, then $(E,D)$ comes from a Higgs bundle.
\end{proposition}
\par Suppose $(E,D)$ be a flat bundle comeing from a Higgs bundle $(E,\ol\pa_E,\theta)$ via a harmonic metric $H$, then we call $(E,D,H)$ is a
harmonic bundle. Denote by $H_{DR}^\ast(X,E)$ and $H_{Dol}^\ast(X,E)$ be the set of $L^\infty$-elements in the cohomology groups of the complexes $(\Omega^\ast(E),D)$
and $(\Omega^\ast(E),\ol\pa_{E,\theta})$ respectively.
\begin{lemma}\label{isomorphic}
Assume $(X,\omega)$ be a K\"{a}hler manifold satisfying Assumption \ref{assump1}.
\begin{enumerate}
\item If $(E,D)$ be a flat bundle and $H$ is a harmonic metric, then any $D$-parallel $L^\infty$-section is also $D_H^{0,1}+\psi_H^{1,0}$-parallel.
\item If $(E,\ol\pa_E,\theta)$ be a Higgs bundle and $H$ is a Hermitian-Einstein metric such that the $\sqrt{-1}\Lambda_\omega F_{\ol\pa_E,\theta,H}=0$,
then any $\ol\pa_{E,\theta}$-parallel $L^\infty$-section is also $D_{\ol\pa_E,\theta,H}$-parallel.
\item If $(E,D,H)$ be a harmonic bundle, then $H_{DR}^0(X,E)=H_{Dol}^0(X,E)$.
\end{enumerate}
\end{lemma}
\begin{proof}
The statement $(3)$ follows from $(1)$ and $(2)$,
Firstly suppose $f$ is $D$-parallel, we have
\begin{equation}\begin{split}
\Delta|f|_H^2
&=\sqrt{-1}\Lambda_\omega(D_H^{1,0}D_H^{0,1}f,f)_H
+\sqrt{-1}\Lambda_\omega(D_H^{1,0}f,D_H^{1,0}f)_H
\\&-\sqrt{-1}\Lambda_\omega(D_H^{0,1}f,D_H^{0,1}f)_H
+\sqrt{-1}\Lambda_\omega(f,D_H^{0,1}D_H^{1,0}f)_H
\\&=|D_H^{0,1}f|_H^2+|\psi_H^{1,0}f|^2_H-\sqrt{-1}\Lambda_\omega(D_H^{1,0}\psi_H^{0,1}f,f)_H
-\sqrt{-1}\Lambda_\omega(f,D_H^{0,1}\psi_H^{1,0}f)_H
\\&=|D_H^{0,1}f|_H^2+|\psi_H^{1,0}f|^2_H+\sqrt{-1}\Lambda_\omega(\psi_H^{0,1}D_H^{1,0}f,f)_H
+\sqrt{-1}\Lambda_\omega(f,\psi_H^{1,0}D_H^{0,1}f)_H
\\&=|D_H^{0,1}f|_H^2+|\psi_H^{1,0}f|^2_H-\sqrt{-1}\Lambda_\omega(D_H^{1,0}f,\psi_H^{1,0}f)_H
+\sqrt{-1}\Lambda_\omega(\psi_H^{0,1}f,D_H^{0,1}f)_H
\\&=2|(D_H^{0,1}+\psi_H^{1,0})f|^2_H.
\end{split}\end{equation}
By Assumption \ref{assump1} we have $f$ is $D_H^{0,1}+\psi_H^{1,0}$-parallel.
\par On the other hand, if $f$ is $\ol\pa_{E,\theta}$-parallel, we deduce
\begin{equation}\begin{split}
\Delta|f|_H^2
&=\sqrt{-1}\Lambda_\omega(\pa_Hf,\pa_Hf)_H+\sqrt{-1}\Lambda_\omega(f,\ol\pa_E\pa_Hf)_H
\\&=|\pa_Hf|_H^2+\sqrt{-1}\Lambda_\omega(f,F_{\ol\pa_E,\theta,H}f)_H-\sqrt{-1}\Lambda_\omega(f,[\theta,\theta^{\ast H}]f)_H
\\&=|\pa_{\theta,H}f|_H^2,
\end{split}\end{equation}
and then we see $f$ is $D_{\ol\pa_E,\theta,H}$-parallel.
\end{proof}
Given a vector bundle $E$ over a complex manifold $(X,\omega)$, equipped with a Hermitian metric $H$ and a connection $D$.
The contraction of pseudo-curvature operator can be written as
\begin{equation}\begin{split}
\sqrt{-1}\Lambda_\omega G_H&=\frac{\sqrt{-1}}{4}\Lambda_\omega[D^{0,1}+\delta_H^{0,1},D^{1,0}-\delta_H^{1,0}]
\\&=\frac{\sqrt{-1}}{4}\Lambda_\omega\left(D^2-(\delta_H^{1,0}+D^{0,1})^2+(D^{1,0}+\delta_H^{0,1})^2-\delta_H^2\right),
\end{split}\end{equation}
If $D$ is flat, we see
\begin{equation}\begin{split}\label{sec428}
\sqrt{-1}\Lambda_\omega G_H&=\frac{\sqrt{-1}}{4}\Lambda_\omega\left(-(\delta_H^{1,0}+D^{0,1})^2+(D^{1,0}+\delta_H^{0,1})^2\right)
\\&=\frac{\sqrt{-1}}{4}\Lambda_\omega(-[D^{0,1},\delta_K^{1,0}+h^{-1}\delta_K^{1,0}h]+[\delta_K^{0,1}+h^{-1}\delta_K^{0,1}h,D^{1,0}])
\\&=\sqrt{-1}\Lambda_\omega G_K+
\frac{\sqrt{-1}}{4}\Lambda_\omega D(h^{-1}D_K^ch),
\end{split}\end{equation}
where $h=K^{-1}H$ and $D_K^c=\delta_K^{0,1}-\delta_K^{1,0}$.
\par Finally, we are in the position prove Theorem \ref{thm3}.
\begin{proof}[\textup{\textbf{Proof of Theorem \ref{thm3}}}]
Given $[D]\in\mathcal{M}_{Flat,K,\phi}$, by Theorem \ref{thm1}, we yield a Poisson metric $H$ with
$\det h=1,|h|\in L^\infty, ||Dh||_{L^2}\leq C||\phi||_{L^1}^{\frac{1}{2}}$ for a positive constant $C$ and $h=K^{-1}H$. Then
Proposition \ref{Higgsstructure1} implies $(\ol\pa_E,\theta)=(D_H^{0,1},\psi_H^{\perp,1,0})$ determines an irreducible Higgs structure
with $c_1(E,H)=c_1(E,K)=0$. Moreover, it can be seen $H$ is harmonic with respect to $D_H+\psi_H^\perp$
and due to Lemma \ref{isomorphic}, the assignment $[D]\mapsto[(\ol\pa_E,\theta)]$ is well-defined. Then we compute
\begin{equation}\begin{split}
\sqrt{-1}\Lambda_\omega F_{\ol\pa_E,\theta,K}
&=\sqrt{-1}\Lambda_\omega F_{\ol\pa_E,\theta,H}+\sqrt{-1}\Lambda_\omega\ol\pa_{E,\theta}(h\pa_{\theta,H}h^{-1})
\\&=\sqrt{-1}\Lambda_\omega D_{\ol\pa_E,\theta,H}(h\pa_{\theta,H}h^{-1})-\sqrt{-1}\Lambda_\omega\pa_{\theta,H}(h\pa_{\theta,H}h^{-1})
\\&=-\sqrt{-1}\Lambda_\omega D_{\ol\pa_E,\theta,H}(hD_{\ol\pa_E,\theta,H}^ch^{-1})+\sqrt{-1}\Lambda_\omega D_{\ol\pa_E,\theta,H}(h\ol\pa_{E,\theta}h^{-1})
\\&-\sqrt{-1}\Lambda_\omega\pa_{\theta,H}(h\pa_{\theta,H}h^{-1})
\\&=4\sqrt{-1}\Lambda_\omega G_H^\perp-4\sqrt{-1}\Lambda_\omega G_K^\perp
+\sqrt{-1}\Lambda_\omega D_{\ol\pa_E,\theta,H}(h\ol\pa_{E,\theta}h^{-1})
\\&-\sqrt{-1}\Lambda_\omega\pa_{\theta,H}(h\pa_{\theta,H}h^{-1})
\\&=-2(D_K^\ast\psi_K)^\perp+\sqrt{-1}\Lambda_\omega\pa_{\theta,H}(h\ol\pa_{E,\theta}h^{-1})
\\&+\sqrt{-1}\Lambda_\omega\ol\pa_{E,\theta}(h\ol\pa_{E,\theta}h^{-1})-\sqrt{-1}\Lambda_\omega\pa_{\theta,H}(h\pa_{\theta,H}h^{-1}),
\end{split}\end{equation}
and
\begin{equation}\begin{split}
\sqrt{-1}\Lambda_\omega\pa_{\theta,H}(h\ol\pa_{E,\theta}h^{-1})
&=\sqrt{-1}\Lambda_\omega(\pa_{\theta,H}h\wedge\ol\pa_{E,\theta}h^{-1})
+\sqrt{-1}\Lambda_\omega(\ol\pa_{E,\theta}h\wedge\pa_{\theta,H}h^{-1})
\\&-\sqrt{-1}\Lambda_\omega(\ol\pa_{E,\theta}h\wedge\pa_{\theta,H}h^{-1})
-\sqrt{-1}\Lambda_\omega(h\ol\pa_{E,\theta}\pa_{\theta,H}h^{-1})
\\&=\sqrt{-1}\Lambda_\omega(\pa_{\theta,H}h\wedge\ol\pa_{E,\theta}h^{-1})
+\sqrt{-1}\Lambda_\omega(\ol\pa_{E,\theta}h\wedge\pa_{\theta,H}h^{-1})
\\&-\sqrt{-1}\Lambda_\omega F_{\ol\pa_E,\theta,K}.
\end{split}\end{equation}
where we have used $(\ref{sec428})$. We obtain from above two equalities that
\begin{equation}\begin{split}\label{sec429}
\sqrt{-1}\Lambda_\omega F_{\ol\pa_E,\theta,K}
&=-(D_K^\ast\psi_K)^\perp+\frac{\sqrt{-1}}{2}\Lambda_\omega(\pa_{\theta,H}h\wedge\ol\pa_{E,\theta}h^{-1})
\\&+\frac{\sqrt{-1}}{2}\Lambda_\omega(\ol\pa_{E,\theta}h\wedge\pa_{\theta,H}h^{-1})+\frac{\sqrt{-1}}{2}\ol\pa_{E,\theta}(h\ol\pa_{E,\theta}h^{-1})
\\&-\frac{\sqrt{-1}}{2}\Lambda_\omega\pa_{\theta,H}(h\pa_{\theta,H}h^{-1}),
\end{split}\end{equation}
from which we conclude $|\Lambda_\omega F_{\ol\pa_E,\theta,K}|\in L^1$.
\par For any $\theta$-invariant sub-holomorphic bundle $V\subset E$, we have the orthogonal decomposition $E=V\oplus V^\perp$ with
respect to $K$ and the projection $\pi_K$ onto $V$. We denote by $K_V$, $H_V$ the metrics on $V$ induced by $K$ and $H$.
If $|\ol\pa_{E,\theta}\pi_K|\not\in L^2$, by Chern-Weil formula $(\ref{ChernHiggs})$ and $|\Lambda_\omega F_{\ol\pa_E,\theta,K}|\in L^1$, we know
\begin{equation}\begin{split}\label{sec4210}
\deg_\omega(V,K)=-\infty.
\end{split}\end{equation}
If $|\ol\pa_{E,\theta}\pi_K|\in L^2$, the same reason indicates $\deg_\omega(V,K)$ is finite. It also holds
\begin{equation}\begin{split}
h_V=\pi_K\circ h\circ\pi_K,
\end{split}\end{equation}
\begin{equation}\begin{split}
\ol\pa_{V,\theta_V}h_V=\pi_K\circ\ol\pa_{E,\theta}h\circ\pi_K+\ol\pa_{E,\theta}\pi_K\circ(\id_E-\pi_K)\circ h\circ\pi_K,
\end{split}\end{equation}
where $h_V=K_V^{-1}H_V$ and $\ol\pa_{V,\theta_V}$ is the induced Higgs structure on $V$. Therefore
\begin{equation}\begin{split}
||\tr(h_V^{-1}\pa_{\theta_V,K_V}h_V)||_{L^2}
&\leq||h_V^{-1}||_{L^\infty}||\ol\pa_{V,\theta_V}h_V||_{L^2}
<\infty.
\end{split}\end{equation}
So Simpson's Lemma 5.2 in \cite{Si1988}(or Yau's Lemma in \cite{Ya1976}) applies,
\begin{equation}\begin{split}
\lim\limits_{j\rightarrow\infty}\int_{X_j}\sqrt{-1}\ol\pa\tr(h_V^{-1}\pa_{\theta_V,K_V}h_V)=0,
\end{split}\end{equation}
where $\{X_j\}$ is a sequence of exhaustion subsets.
And
\begin{equation}\begin{split}\label{sec4211}
\lim\limits_{j\rightarrow\infty}\int_{X_j}\sqrt{-1}\tr F_{\ol\pa_V,\theta_V,H_V}
&=\lim\limits_{j\rightarrow\infty}\int_{X_j}
\sqrt{-1}\tr F_{\ol\pa_V,\theta_V,K_V}
\\&+\lim\limits_{j\rightarrow\infty}\int_{X_j}
\sqrt{-1}\tr(\ol\pa_{V,\theta_V}(h_V^{-1}\pa_{\theta_V,K}h_V))
\\&=\deg_\omega(V,K),
\end{split}\end{equation}
then we have
\begin{equation}\begin{split}\label{sec4212}
\lim\limits_{j\rightarrow\infty}\int_{X_j}\sqrt{-1}\tr F_{\ol\pa_V,\theta_V,H_V}
&=\int_{X}\sqrt{-1}\tr F_{\ol\pa_V,\theta_V,H_V}
=\deg_\omega(V,H).
\end{split}\end{equation}
Since $(\ol\pa_E,\theta)$ is $H$-analytically stable, we conclude by $(\ref{sec4210})$, $(\ref{sec4211})$ and $(\ref{sec4212})$
that $(\ol\pa_E,\theta)$ is $K$-analytically stable.
\par Next assume $\tilde{H}$ be another Hermitian-Einstein metric such that
$\det\tilde{h}=1$ and $|\tilde{h}|\in L^\infty$ for $\tilde{h}=K^{-1}\tilde{H}$. Set $\hat{h}=H^{-1}\tilde{H}$, we have $\sqrt{-1}\Lambda_\omega\pa\ol\pa\tr\hat{h}=|\hat{h}^{-\frac{1}{2}}\pa_{\theta,H}\hat{h}|_{H}^2$ and the Assumption \ref{assump1} implies $\pa_{\theta,H}\hat{h}=0$. We let $E=\mathop{\bigoplus}\limits_{j=1}^mE_j$ be the
eigendecomposition of $\hat{h}$, which is orthogonal with respect to $H$ and $\tilde{H}$. Moreover, it holds $\ol\pa_{E,\theta}\pi_j=0$ for the projection
$\pi_j$ onto $E_j$. But $(\ol\pa_E,\theta)$ is $H$-analytically stable, we know $m=1$, $H=\tilde{H}$ and therefore $[(\ol\pa_E,\theta)]\in\mathcal{M}_{Higgs,K}$.
\par Conversely, given $[(\ol\pa_E,\theta)]\in\mathcal{M}_{Higgs,K,\phi}$, by Mochizuki's work \cite{Mo2020} we know that there
exists a Hermitian-Einstin metric $H$ with $\det h=1$, $|h|\in L^\infty$ and $||\ol\pa_{E,\theta}h||_{L^2}\leq C||\phi||_{L^1}^{\frac{1}{2}}$ for
a positive constant $C$, where $h=K^{-1}H$. Moreover, the Hitchin-Simpson connection $D_{\ol\pa_E,\theta,H}$ is flat and irreducible.
Due to Lemma \ref{isomorphic}, the map $[(\ol\pa_E,\theta)]\mapsto[D_{\ol\pa_E,\theta,H}]$ is well-defined. Next as $H$ is harmonic with respect
to $D=D_{\ol\pa_E,\theta,H}$, a similar discussion as that in $(\ref{sec429})$ shows
\begin{equation}\begin{split}
D_K^\ast\psi_K
&=-\sqrt{-1}\Lambda_\omega F_{\ol\pa_E,\theta,K}+\frac{\sqrt{-1}}{2}\Lambda_\omega(\pa_{\theta,H}h\wedge\ol\pa_{E,\theta}h^{-1})
\\&+\frac{\sqrt{-1}}{2}\Lambda_\omega(\ol\pa_{E,\theta}h\wedge\pa_{\theta,H}h^{-1})+\frac{\sqrt{-1}}{2}\ol\pa_{E,\theta}(h\ol\pa_{E,\theta}h^{-1})
\\&-\frac{\sqrt{-1}}{2}\Lambda_\omega\pa_{\theta,H}(h\pa_{\theta,H}h^{-1}).
\end{split}\end{equation}
and therefore $|(D_K^\ast\psi_K)^\perp|=|D_K^\ast\psi_K|\in L^1$. Then proceeding the argument in the proof of the
Theorem \ref{thm1}, we conclude $D_{\ol\pa_E,\theta,H}$ is analytically stable with respect to the background metric $K$ and hence
$[D_{\ol\pa_E,\theta,H}]\in\mathcal{M}_{Flat,K}$.
\end{proof}
\subsection{Vanishing theorems of characteristic classes}
Associated with a vector bundle $(E,D)$, up to some normalization coefficients, the Kamber-Tondeur classes are defined by
\begin{equation}\begin{split}
\alpha_{2k+1}(E,D)=[\tr\psi_H^{2k+1}]\in H^{2k+1}_{DR}(M,\mathbb{R}),
\end{split}\end{equation}
and the class is independent of the choice of fiber metrics, see \cite{BL1995,Da2019,Du1976,KT1974,PZZ2019} for the related study of
these characteristic classes.
\par By Proposition \ref{Higgsstructure2}, we know $\psi_{H}^{1,0}\wedge\psi_H^{1,0}=0$ and following \cite{Ko2017}, for any $k\geq1$ we deduce
\begin{equation}\begin{split}
\tr\psi_H^{2k+1}
&=\tr\left((\psi_H^{1,0}\wedge\psi_H^{0,1})^{2k}\wedge\psi_H^{1,0}\right)+\tr\left((\psi_H^{0,1}\wedge\psi_H^{1,0})^{2k}\wedge\psi_H^{0,1}\right)
\\&=\tr\left(\psi_H^{1,0}\wedge(\psi_H^{1,0}\wedge\psi_H^{0,1})^{2k}\right)+\tr\left(\psi_H^{0,1}\wedge(\psi_H^{0,1}\wedge\psi_H^{1,0})^{2k}\right)
\\&=0.
\end{split}\end{equation}
So we conclude the following vanishing theorem.
\begin{theorem}\label{vanishing}
Let $(X,\omega)$ be a complete K\"{a}hler manifold with bounded Ricci curvature from below and satisfying the Assumptions \ref{assump1}.
Assume $(E,D)$ be a $K$-analytically stable flat bundle over $X$ such that $|(D_K^\ast\psi_K)^\perp|\leq C\phi$ for a constant $C$
and $|\psi_K|\in L^2$. Then the characteristic classes $\alpha_{2k+1}(E,D)$ vanish for all $k\geq1$.
\end{theorem}
We mention that in the compact setting, this result was firstly proved in \cite{Re1995}.


\begin{thebibliography}{1}
\bibitem{Bi1997}O. Biquard, Fibr\'{e}s de Higgs et connexions int\'{e}grables: le cas logarithmique(diviseur lisse). Ann Sci \'{E}cole
Norm Sup (4), 1997, \textbf{30}(1): 41-96.
\bibitem{BB2004}O. Biquard and P. Boalch, Wild non-abelian Hodge theory on curves. Compos Math, 2004, \textbf{140}(1): 179-204.
\bibitem{BGM2020}O. Biquard, O. Garc\'{\i}a-Prada and I. Mundet i Riera, Parabolic Higgs bundles and representations of the fundamental group
of a punctured surface into a real group. Adv Math, 2020, \textbf{372}: 107305.
\bibitem{BL1995}J. M. Bismut and J. Lott, Flat vector bundles, direct images and higher real analytic torsion. J Amer Math Soc,
1995, \textbf{8}(2): 291-363.
\bibitem{BLK2013}I. Biswas, J. Loftin and M. Stemmler, Flat bundles on affine manifolds. Arab J Math (Springer), 2013, \textbf{2}(2): 159-175.
\bibitem{BK2019}I. Biswas and H. Kasuya, Higgs bundles and flat connections over compact Sasakian manifolds. Commun Math Phys, 2021,
\textbf{385}(1): 267-290.
\bibitem{CJY2019}T. C. Collins, A. Jacob and S. T. Yau, Poisson metrics on flat vector bundles over non-compact curves.
Comm Anal Geom, 2019, \textbf{27}(3): 529-597.
\bibitem{Co1988}K. Corlette, Flat G-bundles with canonical metrics. J. Differential Geom, 1988, \textbf{28}(3): 361-382.
\bibitem{Co1992}K. Corlette, Archimedean superrigidity and hyperbolic geometry. Ann of Math (2), 1992, \textbf{135}(1): 165-182.
\bibitem{Co1993}K. Corlette, Nonabelian Hodge theory, In: Differential geometry: geometry in mathematical physics and related topics
(Los Angeles, CA, 1990), 125-144. Proc. Sympos. Pure Math., 54, Part 2, Amer. Math. Soc., Providence, RI, 1993.
\bibitem{Da2019}J. Daniel, On somes characteristic classes of flat bundles in complex geometry. Ann Inst Fourier (Grenoble),
2019, \textbf{69}(2019): 729-751.
\bibitem{DMG2018}G. Daskalopoulos, C. Mese and G. Wilkin, Higgs bundles over cell complexes and representations of finitely presented groups.
Pacific J Math, 2018, \textbf{296}(1): 31-55.
\bibitem{Do1985}S. K. Donaldson, Anti-self-dual Yang-Mills connections over complex algebraic surfaces and stable vector bundles.
Proc London Math Soc (3), 1985, \textbf{50}(1): 1-26.
\bibitem{Do1987}S. K. Donaldson, Twisted harmonic maps and the self-duality equations. Proc London Math Soc (3), 1987, \textbf{55}(1): 127-131.
\bibitem{Do1992}S. K. Donaldson, Boundary value problems for Yang-Mills fields. J. Geom. Phys, 1992, \textbf{8}(1-4): 89-122.
\bibitem{Du1976}J. Dupont, Simplicial de Rham cohomology and characteristic classes of flat bundles. Topology, 1976, \textbf{15}(3): 233-245.
\bibitem{ES1964}J. Eells and J. H. Sampson, Harmonic mappings of Riemannian manifolds. Amer J Math, 1964, \textbf{86}: 109-160.
\bibitem{GT2001}D. Gilbarg and N. S. Trudinger, Elliptic partial differential equations of second order. Reprint of the 1998 edition.
Classics in Mathematics. Springer-Verlag, Berlin, 2001. xiv+517 pp. ISBN: 3-540-41160-7.
\bibitem{Hi1987}N. J. Hitchin, The self-duality equations on a Riemann surface. Proc London Math Soc (3), 1987, \textbf{55}(1): 59-126.
\bibitem{JY1991}J. Jost and S. T. Yau, Harmonic maps and group representations, In: Differential Geometry, 241-259. Pitman Monogr Surveys Pure Appl Math., 52, Longman Sci Tech., Harlow, 1991.
\bibitem{JZ1996}J. Jost and K. Zuo, Harmonic maps and $Sl(r,\mathbb{C})$-representations of fundamental groups of quasiprojective manifolds.
  J Algebraic Geom, 1996, \textbf{5}(1): 77-106.
\bibitem{JZ1997}J. Jost and K. Zuo, Harmonic maps of infinite energy and rigidity results for representations of fundamental groups of
quasiprojective varieties. J. Differential Geom, 1997, \textbf{47}(3): 469-503.
\bibitem{KT1974}F. Kamber and P. Tondeur, Characteristic invariants of foliated bundles. Manuscripta Math, 1974, \textbf{11}: 51-89.
\bibitem{Ko2017}E. Korman, Characteristic classes of Higgs bundles and Reznikov's theorem. Manuscripta Math, 2017, \textbf{152}(3-4): 433-442.
\bibitem{La1991}F. Labourie, Existence d'applications harmoniques tordues \`{a} valeurs dans les vari\'{e}t\'{e}s \`{a} courbure n\'{e}gative. Proc Amer Math Soc, 1991, \textbf{111}(3): 877-882.
\bibitem{Li1996}J. Y. Li, Hitchin's self-duality equations on complete Riemannian manifolds. Math Ann, 1996, \textbf{306}(3): 419-428.
\bibitem{Lo2009}J. Loftin, Affine Hermitian-Einstein metrics. Asian J Math, 2009, \textbf{13}(1): 101-130.
\bibitem{LT1995}M. L\"{u}bke and A. Teleman, The Kobayashi-Hitchin Correspondence. World Scientific Publishing Co., Inc., River Edge, NJ, 1995. x+254 pp. ISBN: 981-02-2168-1.
\bibitem{Mo2006}T. Mochizuki, Kobayashi-Hitchin correspondence for tame harmonic bundles and an application. Ast\'{e}risque No. 309(2006),
viii+117 pp. ISBN: 978-2-85629-226-6.
\bibitem{Mo2007}T. Mochizuki, Asymptotic behaviour of tame harmonic bundles and an application to pure twistor $D$-modules
\uppercase\expandafter{\romannumeral1}. Mem. Amer. Math. Soc. 185(2007), no. 869, xii+324 pp.
\bibitem{Mo2009}T. Mochizuki, Kobayashi-Hitchin correspondence for tame harmonic bundles \uppercase\expandafter{\romannumeral2}. Geom Topol,
2009, \textbf{13}(1): 359-455.
\bibitem{Mo2011}T. Mochizuki, Wild harmonic bundles and wild pure twistor $D$-mudules. Ast\'{e}risque No. 340(2011), x+607 pp. ISBN: 978-2-85629-332-4
\bibitem{Mo2020}T. Mochizuki, Kobayashi-Hitchin correspondence for analytically stable bundles. Trans Amer Math Soc, 2020, \textbf{373}(1): 551-596.
\bibitem{NS1965}M. S. Narasimhan and C. S. Seshadri, Stable and unitary vector bundles on a compact Riemann surface. Ann Math (2), 1965, \textbf{82}:
540-567.
\bibitem{PZZ2019}C. P. Pan, C. J. Zhang and X. Zhang, Projectively flat bundles and semi-stable Higgs bundles. arXiv: 1911.03593, 2019.
\bibitem{Re1995}A. Reznikov, All regulators of flat bundles are torsion. Ann Math (2), 1995, \textbf{141}(2): 373-386.
\bibitem{Sa1986}J. H. Sampson, Applications of harmonic maps to K\"{a}hler geometry,
In: Complex differential geometry and nonlinear differential equations(Brunswick, Maine, 1984), 125-134.
Contemp. Math., 49, Amer. Math. Soc., Providence, RI, 1986.
\bibitem{SZZ2019}Z. H. Shen, C. J. Zhang and X. Zhang, Flat Higgs bundles over non-compact affine Gauduchon manifolds. arXiv: 1909.12577, 2019.
\bibitem{Si1988}C. T. Simpson, Constructing variations of Hodge structure using Yang-Mills theory and applications to uniformization. J Amer Math Soc, 1988, \textbf{1}(4): 867-918.
\bibitem{Si1990}C. T. Simpson, Harmonic bundles on noncompact curves. J Amer Math Soc, 1990, \textbf{3}(3): 713-770.
\bibitem{Si1992}C. T. Simpson, Higgs bundles and local systems. Inst Hautes \'{E}tudes Sci Publ Math, 1992, No. 75: 5-95.
\bibitem{Si1980}Y. T. Siu, The complex-analyticity of harmonic maps and strong rigidity of complex K\"{a}hler manifolds. Ann of Math (2), 1980, \textbf{112}(2): 73-111.
\bibitem{Uh1982}K. Uhlenbeck, Connections with $L^p$ bounds on curvature. Comm Math Phys, 1982, \textbf{83}(1): 31-42.
\bibitem{UY1986}K. Uhlenbeck and S. T. Yau, On the existence of Hermitian-Yang-Mills connections in stable vector bundles. Comm Pure Appl Math, 1986, \textbf{39}(S): 257-293.
\bibitem{WZ2011}Y. Wang and X. Zhang, Twisted holomorphic chains and vortex equations over non-compact K\"{a}hler manifolds. J Math Anal Appl, 2011,
\textbf{373}(1): 179-202.
\bibitem{Ya1976}S. T. Yau, Some function-theoretic properties of complete Riemannian manifold and their applications to geometry. Indiana Univ Math
J, 1976, \textbf{25}(7): 659-670.
\bibitem{Zh2005}X. Zhang, Hermitian-Einstein metrics on holomorphic vector bundles over Hermitian manifolds. J Geom Phys, 2005, \textbf{53}(3): 315-335.
\bibitem{Zh2021}X. Zhang, The limit of the harmonic flow on flat complex vector bundle. arXiv: 2101.07443, 2021.
\end{thebibliography}
\end{document}